\documentclass[a4paper,10pt,fleqn]{article}
\usepackage[latin1]{inputenc}
\usepackage[T1]{fontenc}
\usepackage[left=1.5cm,right=1.5cm,bottom=2cm,top=2cm]{geometry}

\usepackage[all,2cell,v2,arrow, curve, frame, matrix, graph, tips]{xy}

\UseAllTwocells

\usepackage[square,numbers,sort]{natbib}
\usepackage{tikz-cd}
\usepackage{ifthen}
\usepackage{mathbbol}
\usepackage{savesym}
\usepackage{mathtools}
\usepackage{amsmath}
\usepackage{amssymb}
\usepackage{tcolorbox}
\usepackage{enumerate}
\usepackage[thmmarks,standard]{ntheorem}
\usepackage{microtype}
\usepackage{stmaryrd}

\usepackage{verbatim}
\usepackage[bookmarks=true,
            bookmarksnumbered=true, breaklinks=true,
            pdfstartview=FitH, hyperfigures=false,
            plainpages=false, naturalnames=true,
            colorlinks=true,pagebackref=true,
            pdfpagelabels]{hyperref}
\usepackage{cleveref}
\usepackage[toc,page]{appendix}
\theoremstyle{plain}

\newcommand*{\E}{\ensuremath{\mathbb{S}\text{ets}}}
\newcommand*{\CAT}{\ensuremath{\mathbb{C}\text{AT}}}

\newcommand*{\Cat}{\ensuremath{\mathbb{C}\text{at}}}

\newcommand*{\Glob}{\operatorname{\mathbb{G}\mathrm{lob}}}

\title{Globular weak $(n,\infty)$-Transformations ($n\in\mathbb{N}$) in the sense of Grothendieck}
\author{Camell
  Kachour}
  
\begin{document}
\maketitle
\vspace*{3.5cm}
\begin{abstract}
This article describe globular weak $(n,\infty)$-transformations ($n\in\mathbb{N}$) in the sense of Grothendieck, i.e for each
$n\in\mathbb{N}$ we build a coherator $\Theta^{\infty}_{\mathbb{M}^n}$ which sets models are globular weak $(n,\infty)$-transformations.
A natural globular filtration emerges from these coherators. 
\end{abstract}

\begin{minipage}{118mm}{\small
      }\end{minipage}

\hypersetup{%
  linkcolor=blue}%
\tableofcontents
\vspace*{1cm}

\vspace*{1cm}

\section*{Introduction}   

We start this article by defining, for each $n\in\mathbb{N}^*$ a coherator $\Theta^{\infty}_{\mathbb{M}^n}$ in the sense of Grothendieck
(see \cite{grothendieck,malts-cat1,malts-cat,malts-gr}), which is a specific sketch \cite{lair-coppey,makkai-pare} such that models of it are 
globular weak $(n,\infty)$-transformations. 

Recently John Bourke \cite{bourke-injectif} proved a conjecture of Dimitri Ara \cite{ara-these}, which shows
that globular weak $\infty$-categories in the sense of Batanin \cite{batanin-main} and globular weak $\infty$-categories in the sense 
of Grothendieck \cite{grothendieck,malts-cat} are equivalent: more precisely globular weak $\infty$-categories in the sense of Batanin are
$\mathbb{B}^0_C$-algebras where $\mathbb{B}^0_C$ is a specific globular operad, and globular weak 
$\infty$-categories in the sense of Grothendieck are models of a specific theory $\Theta^{\infty}_{\mathbb{M}^0}$; it is 
proved in \cite{clemens,mark-nerve} that the monad associated to any globular operads is strongly cartesian \cite{mark-nerve} thus
for the operad $\mathbb{B}^0_C$ it
leads to a nice theory\footnote{By \textit{theory} here we mean a small category equipped with a chosen
set of projective cones \cite{lair-coppey,makkai-pare}.} $\Theta_{\mathbb{B}^0_C}$ which is Morita equivalent to $\Theta^{\infty}_{\mathbb{M}^0}$
\cite{bourke-injectif,ara-these} in the sense that their category of models are equivalent. 

Thus we suspect that models of the coherators $\Theta^{\infty}_{\mathbb{M}^n}$ ($n\in\mathbb{N}^*$) build in this article 
should be also weak $(n,\infty)$-transformations in the sense of \cite{cam-cgasa3,cam-phd} where they are defined as
$\mathbb{B}^n_C$-algebras ($n\in\mathbb{N}^*$) for specific globular operads $\mathbb{B}^n_C$: similarly to
\cite{bourke-injectif}, for each integers $n\geq 1$, the theory $\Theta^{\infty}_{\mathbb{M}^n}$ and the theory
$\Theta_{\mathbb{B}^n_C}$ should be Morita equivalent. 

Finally we show that $\Theta^{\infty}_{\mathbb{M}^2}$-models in $\E$ of dimension $2$, i.e globular weak $(2,\infty)$-natural 
transformations of dimension $2$, are pseudo-$2$-natural transformations. 

The combinatorics developed in this article lead to an accurate description of the monad $\mathbb{T}^n$ (for each $n\in\mathbb{N}^*$) defined 
on the cartesian product of the 
category of globular sets with itself, which algebras are globular strict $(n,\infty)$-transformations;
this is similar to the description of the monad $\mathbb{S}$ on the category of precubical sets as described in \cite{cam-cubic} which algebras are
cubical strict $\infty$-categories with connections. We shall
give more precisions in a later version of this article. In particular it will show that these monads $\mathbb{T}^n$ (for each $n\in\mathbb{N}^*$)
are cartesian as conjectured in \cite{cam-cgasa3}. We shall also use the formalism developed in \cite{bourke-garner} to show that the coherators $\Theta^{\infty}_{\mathbb{M}^n}$ ($n\in\mathbb{N}^*$) lead to monads on the cartesian product of the category of globular sets with itself, which algebras correspond to these globular weak $(n,\infty)$-transformations in the sense of Grothendieck. 

{\bf Acknowledgement.}
In the memory of Alexandre Grothendieck.

\section{Globular theories $\Theta_{\mathbb{M}^n}$ of $(n,\infty)$-magmas ($n\in\mathbb{N}$).}
\label{1}

\subsection{The globular $\mathbb{G}_n$-extensions $\Theta_n$ ($n\in\mathbb{N}$) and their filtration.}
\label{2}

The $n$-\textit{globe categories} $\mathbb{G}_n$ ($n\in\mathbb{N}$) are small categories which objects are
the basic arities (or atomic arities) of operations behind the structure of globular higher transformations. Their
accurate descriptions (see below) is of first importance for theoretical reasons : their 
\textit{globular completions} $\Theta_n$ ($n\in\mathbb{N}$) defined just below, have as objects the arities of operations behind the 
structure of globular higher transformations, and the important filtration $\Theta_{\bullet}$ described below 
is easily obtain from the filtration $\mathbb{G}_{\bullet}$ below, and the universal property of the completion
$\Theta_n$ of each $\mathbb{G}_{n}$.

The $0$-\textit{globe category} is a small category $\mathbb{G}_0$ which objects are formal symbols
$1(n)$ ($n\in\mathbb{N}$) which arrows are generate by maps $s^{n}_{n-1}, t^{n}_{n-1}$ 

$$\begin{tikzcd}
 1(0)
 \arrow[rr, yshift=1.5ex,"s^{1}_{0}"]\arrow[rr, yshift=-1.5ex,"t^{1}_{0}"{below}]
 &&1(1)
\arrow[rr, yshift=1.5ex,"s^{2}_{1}"]\arrow[rr, yshift=-1.5ex,"t^{2}_{1}"{below}]  
&& 1(2) 
\arrow[rr, yshift=1.5ex,"s^{3}_{2}"]\arrow[rr, yshift=-1.5ex,"t^{3}_{2}"{below}]
&&1(3) 
\arrow[rr, yshift=1.5ex,"s^{4}_{3}"]\arrow[rr, yshift=-1.5ex,"t^{4}_{3}"{below}]
&&1(4)\cdots 1(n-1)  
\arrow[rr, yshift=1.5ex,"s^{n}_{n-1}"]\arrow[rr, yshift=-1.5ex,"t^{n}_{n-1}"{below}]
&& 1(n)\cdots    
\end{tikzcd}$$
which follows the \textit{globular equations}
$s^{n}_{n-1}s^{n-1}_{n-2}=t^{n}_{n-1}s^{n-1}_{n-2}$ and $t^{n}_{n-1}t^{n-1}_{n-2}=s^{n}_{n-1}t^{n-1}_{n-2}$
for all integers $n\geq 2$. The maps $s^{n}_{n-1}$ are called \textit{cosourses} and
the maps $t^{n}_{n-1}$ are called \textit{cotargets}. The composition
$s^{n}_{n-1}s^{n-1}_{n-2}\cdots s^{p+1}_{p}$ is denoted $s^{n}_{p}$ and the composition
$t^{n}_{n-1}t^{n-1}_{n-2}\cdots t^{p+1}_{p}$ is denoted $t^{n}_{p}$.

\begin{definition}
A $\mathbb{G}_0$-presheaf i.e an object of $[\mathbb{G};\E]$ is called a \textit{globular filtration}, whereas a 
$\mathbb{G}_0^{\text{op}}$-presheaf i.e an object of $[\mathbb{G}^{\text{op}};\E]$ is called a \textit{globular set}.
\end{definition}

Let $\mathbb{S}^0$ the monad on $[\mathbb{G}_0^{\text{op}};\E]$ of globular strict $\infty$-categories. The category
$\big(1\downarrow\mathbb{S}^0(1)\big)$ of elements of the presheaf $\mathbb{S}^0(1)$ is commonly denoted
$\Theta_0$, and plays an important role for describing models for theories for globular higher structure (for
example the theory $\Theta_{\text{S}^0}$ which set-models are globular strict $\infty$-categories, or the theory $\Theta_{\text{W}^0}$
which set-models are globular weak $\infty$-categories, etc.), those based on monads on $[\mathbb{G}_0^{\text{op}};\E]$. 
Objects of $\Theta_0$ are globular trees in the sense of \cite{batanin-main}
and they play the role of arities for operations inside such theories. Actually we can avoid the use of the monad 
$\mathbb{S}^0$ to build $\Theta_0$ by using specific colimits based on the globe category $\mathbb{G}_0$:
first we consider tables $t$ of non-negative integers:

$$\begin{pmatrix} 
i_1&&&&i_2&&&&i_3&\cdot&\cdot&\cdot&i_{k-1}&&&&i_{k}\\\\
&&i'_1&&&&i'_2&&\cdot&\cdot&\cdot&\cdot&\cdot&&i'_{k-1}
\end{pmatrix}$$

where $k\geq 1$, $i_l>i'_l<i_{l+1}$ and $1\leq l\leq k-1$. 

Let $\mathcal{C}$ a category and let 
\begin{tikzcd}
\mathbb{G}_0\arrow[rr,"F"]&&\mathcal{C}
\end{tikzcd}
a functor. We denote $F(1(n))=D^n$ and we shall keep the same 
notations for the image of cosources : $F(s^{i_l}_{i_{l'}})=s^{i_l}_{i_{l'}}$, 
and for the image of cotargets : $F(t^{i_l}_{i_{l'}})=t^{i_l}_{i_{l'}}$, because
no risk of confusion will occur. In this case 
\begin{tikzcd}
\mathbb{G}_0\arrow[rr,"F"]&&\mathcal{C}
\end{tikzcd}
is called a globular $\mathbb{G}_0$-extension if for all tables $t$ as just above, the
colimit of the following diagram:

$$\begin{tikzcd}
D^{i_1}&&D^{i_2}&&D^{i_3}&&\cdots&&D^{i_{k-1}}&&D^{i_k}\\
&D^{i'_1}\arrow[lu,"t^{i_1}_{i'_1}"]\arrow[ru,"s^{i_2}_{i'_1}"]
&&D^{i'_2}\arrow[lu,"t^{i_2}_{i'_2}"]\arrow[ru,"s^{i_3}_{i'_2}"]&&
D^{i'_3}\arrow[lu]&\cdots&
D^{i'_{k-2}}\arrow[ru]&&
D^{i'_{k-1}}\arrow[lu,"t^{i_{k-1}}_{i'_{k-1}}"]\arrow[ru,"s^{i_k}_{i'_{k-1}}"]
\end{tikzcd}$$

exists in $\mathcal{C}$.
In \cite{grothendieck} Alexander Grothendieck calls these colimits \textit{globular sums}. We 
prefer call these colimits \textit{globular $\mathbb{G}_0$-sums}.
A morphism of globular $\mathbb{G}_0$-extensions, also called \textit{globular $\mathbb{G}_0$-functor}, is 
given by a commutative triangle in $\CAT$:

$$\begin{tikzcd}
&&\mathcal{C}\arrow[dd,"H"]\\
\mathbb{G}_0\arrow[rru,"F"]\arrow[rrd,"F'"{below}]\\
&&\mathcal{C}'
\end{tikzcd}$$

such that the functor $H$ preserves
globular $\mathbb{G}_0$-sums. The category of globular $\mathbb{G}_0$-extensions
is denoted $\mathbb{G}_0\text{-}\mathbb{E}\text{xt}$. An initial object
of $\mathbb{G}_0\text{-}\mathbb{E}\text{xt}$: 
\begin{tikzcd}
\mathbb{G}_0\arrow[rr,"i"]&&\Theta_0
\end{tikzcd}, gives the small category $\Theta_0$. Tables as above are
another formulation of trees in the sense of \cite{batanin-main}, thus
such tables are preferably called \textit{globular $\mathbb{G}_0$-trees}.
This small category $\Theta_0$ can also be described as the full subcategory
of $\mathbb{G}\text{lob}$ which objects are globular $\mathbb{G}_0$-trees.

Now we are going to extend this $\Theta_0$ for globular strict $\infty$-functors: this
analogue shall be denoted $\Theta_1$. But also globular strict 
$\infty$-natural transformations have their own "$\Theta$" that we denote
$\Theta_2$. Let us put some terminology: globular strict $(0,\infty)$-transformations
are globular strict $\infty$-categories; globular strict $(1,\infty)$-transformations
are globular strict $\infty$-functors; globular strict $(2,\infty)$-transformations
are globular strict $\infty$-natural transformations; globular strict $(3,\infty)$-transformations
are globular strict $\infty$-modifications, and so on. For all integers $n\in\mathbb{N}$ 
the definition of globular strict $(n,\infty)$-transformations is in \cite{cam-cgasa3}.

Let us first describe the small categories $\mathbb{G}_1$, $\mathbb{G}_2$ and $\mathbb{G}_3$ which 
shall be used to get respectively $\Theta_1$, $\Theta_2$ and $\Theta_3$:

The $1$-\textit{globe category} is a small category $\mathbb{G}_1$ which objects are formal symbols
$1(n)$, $f^n(1(n))$, $2(n)$ ($n\in\mathbb{N}$), which arrows are generate by maps $s^{n}_{n-1}, t^{n}_{n-1}$ 
which follows the globular equations:
$s^{n}_{n-1}s^{n-1}_{n-2}=t^{n}_{n-1}s^{n-1}_{n-2}$ and $t^{n}_{n-1}t^{n-1}_{n-2}=s^{n}_{n-1}t^{n-1}_{n-2}$
for all integers $n\geq 2$:

$$\begin{tikzcd}
 1(0)
 \arrow[rr, yshift=1.5ex,"s^{1}_{0}"]\arrow[rr, yshift=-1.5ex,"t^{1}_{0}"{below}]
 &&1(1)
\arrow[rr, yshift=1.5ex,"s^{2}_{1}"]\arrow[rr, yshift=-1.5ex,"t^{2}_{1}"{below}]  
&& 1(2) 
\arrow[rr, yshift=1.5ex,"s^{3}_{2}"]\arrow[rr, yshift=-1.5ex,"t^{3}_{2}"{below}]
&&1(3) 
\arrow[rr, yshift=1.5ex,"s^{4}_{3}"]\arrow[rr, yshift=-1.5ex,"t^{4}_{3}"{below}]
&&1(4)\cdots 1(n-1)  
\arrow[rr, yshift=1.5ex,"s^{n}_{n-1}"]\arrow[rr, yshift=-1.5ex,"t^{n}_{n-1}"{below}]
&& 1(n)\cdots    
\end{tikzcd}$$

$$\begin{tikzcd}
 f^{0}(1(0))
 \arrow[r, yshift=1.5ex,"s^{1}_{0}"]\arrow[r, yshift=-1.5ex,"t^{1}_{0}"{below}]
 &f^{1}(1(1))
\arrow[r, yshift=1.5ex,"s^{2}_{1}"]\arrow[r, yshift=-1.5ex,"t^{2}_{1}"{below}]  
& f^{2}(1(2)) 
\arrow[r, yshift=1.5ex,"s^{3}_{2}"]\arrow[r, yshift=-1.5ex,"t^{3}_{2}"{below}]
&f^{3}(1(3)) 
\arrow[r, yshift=1.5ex,"s^{4}_{3}"]\arrow[r, yshift=-1.5ex,"t^{4}_{3}"{below}]
&f^{4}(1(4))\cdots f^{n-1}(1(n-1))  
\arrow[r, yshift=1.5ex,"s^{n}_{n-1}"]\arrow[r, yshift=-1.5ex,"t^{n}_{n-1}"{below}]
& f^{n}(1(n))\cdots    
\end{tikzcd}$$
$$\begin{tikzcd}
 2(0)
 \arrow[rr, yshift=1.5ex,"s^{1}_{0}"]\arrow[rr, yshift=-1.5ex,"t^{1}_{0}"{below}]
 &&2(1)
\arrow[rr, yshift=1.5ex,"s^{2}_{1}"]\arrow[rr, yshift=-1.5ex,"t^{2}_{1}"{below}]  
&& 2(2) 
\arrow[rr, yshift=1.5ex,"s^{3}_{2}"]\arrow[rr, yshift=-1.5ex,"t^{3}_{2}"{below}]
&&2(3) 
\arrow[rr, yshift=1.5ex,"s^{4}_{3}"]\arrow[rr, yshift=-1.5ex,"t^{4}_{3}"{below}]
&&2(4)\cdots 2(n-1)  
\arrow[rr, yshift=1.5ex,"s^{n}_{n-1}"]\arrow[rr, yshift=-1.5ex,"t^{n}_{n-1}"{below}]
&& 2(n)\cdots    
\end{tikzcd}$$

which follow the globular equations as just above.

We can see that $\mathbb{G}_1$ has three subcategories denoted respectively by $\mathbb{G}_{1,0}$, 
$\mathbb{G}_{1,t}$ and $\mathbb{G}_{1,1}$, and which are connected. Also 
we deliberately use the same notations for the cosources and cotargets maps of these three
subcategories because no confusions will occur. The notations that we used for objects of $\mathbb{G}_1$ shall be suggestive
when we will give an accurate description of the monad $\mathbb{S}^1$ of globular strict $\infty$-functors.

The $2$-\textit{globe category} is a small category $\mathbb{G}_2$ which objects are formal symbols
$1(n)$, $f^n(1(n))$, $\xi_1$, $g^n(1(n))$, $2(n)$ ($n\in\mathbb{N}$), which arrows are generate by maps $s^{n}_{n-1}, t^{n}_{n-1}$ 
which follows the globular equations:
$s^{n}_{n-1}s^{n-1}_{n-2}=t^{n}_{n-1}s^{n-1}_{n-2}$ and $t^{n}_{n-1}t^{n-1}_{n-2}=s^{n}_{n-1}t^{n-1}_{n-2}$
for all integers $n\geq 2$:

$$\begin{tikzcd}
 1(0)
 \arrow[rr, yshift=1.5ex,"s^{1}_{0}"]\arrow[rr, yshift=-1.5ex,"t^{1}_{0}"{below}]
 &&1(1)
\arrow[rr, yshift=1.5ex,"s^{2}_{1}"]\arrow[rr, yshift=-1.5ex,"t^{2}_{1}"{below}]  
&& 1(2) 
\arrow[rr, yshift=1.5ex,"s^{3}_{2}"]\arrow[rr, yshift=-1.5ex,"t^{3}_{2}"{below}]
&&1(3) 
\arrow[rr, yshift=1.5ex,"s^{4}_{3}"]\arrow[rr, yshift=-1.5ex,"t^{4}_{3}"{below}]
&&1(4)\cdots 1(n-1)  
\arrow[rr, yshift=1.5ex,"s^{n}_{n-1}"]\arrow[rr, yshift=-1.5ex,"t^{n}_{n-1}"{below}]
&& 1(n)\cdots    
\end{tikzcd}$$

$$\begin{tikzcd}
 f^{0}(1(0))\arrow[rd,"s^{1}_{0}", near end]
 \arrow[r, yshift=1.5ex,"s^{1}_{0}"]\arrow[r, yshift=-1.5ex,"t^{1}_{0}"]
 &f^{1}(1(1))
\arrow[r, yshift=1.5ex,"s^{2}_{1}"]\arrow[r, yshift=-1.5ex,"t^{2}_{1}"{below}]  
& f^{2}(1(2)) 
\arrow[r, yshift=1.5ex,"s^{3}_{2}"]\arrow[r, yshift=-1.5ex,"t^{3}_{2}"{below}]
&f^{3}(1(3)) 
\arrow[r, yshift=1.5ex,"s^{4}_{3}"]\arrow[r, yshift=-1.5ex,"t^{4}_{3}"{below}]
&f^{4}(1(4))\cdots f^{n-1}(1(n-1))  
\arrow[r, yshift=1.5ex,"s^{n}_{n-1}"]\arrow[r, yshift=-1.5ex,"t^{n}_{n-1}"{below}]
& f^{n}(1(n))\cdots\\ 
&\xi_1\\  
g^{0}(1(0))\arrow[ru,"t^{1}_{0}"{below},near end]
 \arrow[r, yshift=1.5ex,"s^{1}_{0}"{below}]\arrow[r, yshift=-1.5ex,"t^{1}_{0}"{below}]
 &g^{1}(1(1))
\arrow[r, yshift=1.5ex,"s^{2}_{1}"]\arrow[r, yshift=-1.5ex,"t^{2}_{1}"{below}]  
& g^{2}(1(2)) 
\arrow[r, yshift=1.5ex,"s^{3}_{2}"]\arrow[r, yshift=-1.5ex,"t^{3}_{2}"{below}]
&g^{3}(1(3)) 
\arrow[r, yshift=1.5ex,"s^{4}_{3}"]\arrow[r, yshift=-1.5ex,"t^{4}_{3}"{below}]
&g^{4}(1(4))\cdots 
g^{n-1}(1(n-1))  
\arrow[r, yshift=1.5ex,"s^{n}_{n-1}"]\arrow[r, yshift=-1.5ex,"t^{n}_{n-1}"{below}]
& g^{n}(1(n))\cdots  
 \end{tikzcd}$$
 
 $$\begin{tikzcd}
 2(0)
 \arrow[rr, yshift=1.5ex,"s^{1}_{0}"]\arrow[rr, yshift=-1.5ex,"t^{1}_{0}"{below}]
 &&2(1)
\arrow[rr, yshift=1.5ex,"s^{2}_{1}"]\arrow[rr, yshift=-1.5ex,"t^{2}_{1}"{below}]  
&& 2(2) 
\arrow[rr, yshift=1.5ex,"s^{3}_{2}"]\arrow[rr, yshift=-1.5ex,"t^{3}_{2}"{below}]
&&2(3) 
\arrow[rr, yshift=1.5ex,"s^{4}_{3}"]\arrow[rr, yshift=-1.5ex,"t^{4}_{3}"{below}]
&&2(4)\cdots 2(n-1)  
\arrow[rr, yshift=1.5ex,"s^{n}_{n-1}"]\arrow[rr, yshift=-1.5ex,"t^{n}_{n-1}"{below}]
&& 2(n)\cdots    
\end{tikzcd}$$

The $3$-\textit{globe category} is a small category $\mathbb{G}_3$ which objects are formal symbols
$1(n)$, $f^n(1(n))$, $\alpha_1(1(0))$, $\xi_2(1(0))$, $\beta_1(1(0))$, $g^n(1(n))$, $2(n)$ ($n\in\mathbb{N}$), which arrows are 
generate by maps $s^{n}_{n-1}, t^{n}_{n-1}$ which follows the globular equations:
$s^{n}_{n-1}s^{n-1}_{n-2}=t^{n}_{n-1}s^{n-1}_{n-2}$ and $t^{n}_{n-1}t^{n-1}_{n-2}=s^{n}_{n-1}t^{n-1}_{n-2}$
for all integers $n\geq 2$:
 
 $$\begin{tikzcd}
 1(0)
 \arrow[rr, yshift=1.5ex,"s^{1}_{0}"]\arrow[rr, yshift=-1.5ex,"t^{1}_{0}"{below}]
 &&1(1)
\arrow[rr, yshift=1.5ex,"s^{2}_{1}"]\arrow[rr, yshift=-1.5ex,"t^{2}_{1}"{below}]  
&& 1(2) 
\arrow[rr, yshift=1.5ex,"s^{3}_{2}"]\arrow[rr, yshift=-1.5ex,"t^{3}_{2}"{below}]
&&1(3) 
\arrow[rr, yshift=1.5ex,"s^{4}_{3}"]\arrow[rr, yshift=-1.5ex,"t^{4}_{3}"{below}]
&&1(4)\cdots 1(n-1)  
\arrow[rr, yshift=1.5ex,"s^{n}_{n-1}"]\arrow[rr, yshift=-1.5ex,"t^{n}_{n-1}"{below}]
&& 1(n)\cdots    
\end{tikzcd}$$
 
  $$\begin{tikzcd}
 f^{0}(1(0))\arrow[rddd,"t^{1}_{0}"{left}, near start]\arrow[rd,"s^{1}_{0}"{below},near end]
 \arrow[r, yshift=1.5ex,"s^{1}_{0}"]\arrow[r, yshift=-1.5ex,"t^{1}_{0}"]
 &f^{1}(1(1))
\arrow[r, yshift=1.5ex,"s^{2}_{1}"]\arrow[r, yshift=-1.5ex,"t^{2}_{1}"{below}]  
& f^{2}(1(2)) 
\arrow[r, yshift=1.5ex,"s^{3}_{2}"]\arrow[r, yshift=-1.5ex,"t^{3}_{2}"{below}]
&f^{3}(1(3)) 
\arrow[r, yshift=1.5ex,"s^{4}_{3}"]\arrow[r, yshift=-1.5ex,"t^{4}_{3}"{below}]
&f^{4}(1(4))\cdots f^{n-1}(1(n-1))  
\arrow[r, yshift=1.5ex,"s^{n}_{n-1}"]\arrow[r, yshift=-1.5ex,"t^{n}_{n-1}"{below}]
& f^{n}(1(n))\cdots\\ 
&\alpha_1(1(0))\arrow[rd,"s^{2}_{1}"]\\
&&\xi_2(1(0))\\
&\beta_1(1(0))\arrow[ru,"t^{2}_{1}"{below}]\\
 g^{0}(1(0))\arrow[ruuu,"s^{1}_{0}"{left}, near start]\arrow[ru,"t^{1}_{0}", near end]
 \arrow[r, yshift=1.5ex,"s^{1}_{0}"{below}]\arrow[r, yshift=-1.5ex,"t^{1}_{0}"{below}]
 &g^{1}(1(1))
\arrow[r, yshift=1.5ex,"s^{2}_{1}"]\arrow[r, yshift=-1.5ex,"t^{2}_{1}"{below}]  
& g^{2}(1(2)) 
\arrow[r, yshift=1.5ex,"s^{3}_{2}"]\arrow[r, yshift=-1.5ex,"t^{3}_{2}"{below}]
&g^{3}(1(3)) 
\arrow[r, yshift=1.5ex,"s^{4}_{3}"]\arrow[r, yshift=-1.5ex,"t^{4}_{3}"{below}]
&g^{4}(1(4))\cdots 
g^{n-1}(1(n-1))  
\arrow[r, yshift=1.5ex,"s^{n}_{n-1}"]\arrow[r, yshift=-1.5ex,"t^{n}_{n-1}"{below}]
& g^{n}(1(n))\cdots     
\end{tikzcd}$$

$$\begin{tikzcd}
 2(0)
 \arrow[rr, yshift=1.5ex,"s^{1}_{0}"]\arrow[rr, yshift=-1.5ex,"t^{1}_{0}"{below}]
 &&2(1)
\arrow[rr, yshift=1.5ex,"s^{2}_{1}"]\arrow[rr, yshift=-1.5ex,"t^{2}_{1}"{below}]  
&& 2(2) 
\arrow[rr, yshift=1.5ex,"s^{3}_{2}"]\arrow[rr, yshift=-1.5ex,"t^{3}_{2}"{below}]
&&2(3) 
\arrow[rr, yshift=1.5ex,"s^{4}_{3}"]\arrow[rr, yshift=-1.5ex,"t^{4}_{3}"{below}]
&&2(4)\cdots 2(n-1)  
\arrow[rr, yshift=1.5ex,"s^{n}_{n-1}"]\arrow[rr, yshift=-1.5ex,"t^{n}_{n-1}"{below}]
&& 2(n)\cdots    
\end{tikzcd}$$

Let $\mathbb{S}^n$ the monad on $[\mathbb{G}_0^{\text{op}}+\mathbb{G}_0^{\text{op}};\E]$ of globular strict 
$(n,\infty)$-transformations. The category $\big(1+2\downarrow\mathbb{S}^n(1+2)\big)$ of elements of the 
presheaf $\mathbb{S}^n(1+2)$ is denoted $\Theta_n$, and plays an important role for describing models for 
theories for globular higher structures like the theory $\Theta_{\text{S}^n}$ which set-models are globular strict 
$(n,\infty)$-transformations, or the theory $\Theta_{\text{W}^n}$ which set-models are globular weak $(n,\infty)$-transformations 
defined in \cite{cam-cgasa3}. Now it is interesting to follow the steps of Grothendieck and to construct $\Theta_n$ 
with adapted colimits, avoiding the use of the monad $\mathbb{S}^n$. 
For that we start with the following small category $\mathbb{G}_{n}$ (for $n\geq 4$; $\mathbb{G}_{1}$, $\mathbb{G}_{2}$, $\mathbb{G}_{3}$ 
were described above), that we call the $n$-\textit{globe category}: it is a small category which objects are formal symbols
$1(n)$, $f^n(1(n))$, $\alpha_{l}(1(0))$, $\xi_n(1(0))$, $\beta_{l}(1(0))$, $g^n(1(n))$, $2(n)$ 
($n\in\mathbb{N}$, and $1\leq l\leq n-1$ if $n>1$), which arrows are generate by maps $s^{n}_{n-1}, t^{n}_{n-1}$ 
which follows the globular equations:
$s^{n}_{n-1}s^{n-1}_{n-2}=t^{n}_{n-1}s^{n-1}_{n-2}$ and $t^{n}_{n-1}t^{n-1}_{n-2}=s^{n}_{n-1}t^{n-1}_{n-2}$
for all integers $n\geq 2$:

$$\begin{tikzcd}
 1(0)
 \arrow[rr, yshift=1.5ex,"s^{1}_{0}"]\arrow[rr, yshift=-1.5ex,"t^{1}_{0}"{below}]
 &&1(1)
\arrow[rr, yshift=1.5ex,"s^{2}_{1}"]\arrow[rr, yshift=-1.5ex,"t^{2}_{1}"{below}]  
&& 1(2) 
\arrow[rr, yshift=1.5ex,"s^{3}_{2}"]\arrow[rr, yshift=-1.5ex,"t^{3}_{2}"{below}]
&&1(3) 
\arrow[rr, yshift=1.5ex,"s^{4}_{3}"]\arrow[rr, yshift=-1.5ex,"t^{4}_{3}"{below}]
&&1(4)\cdots 1(n-1)  
\arrow[rr, yshift=1.5ex,"s^{n}_{n-1}"]\arrow[rr, yshift=-1.5ex,"t^{n}_{n-1}"{below}]
&& 1(n)\cdots    
\end{tikzcd}$$
 
$$\begin{tikzcd}
 f^{0}(1(0))\arrow[rddd,"t^{1}_{0}"{left}, near start]\arrow[rd,"s^{1}_{0}"{below}]
 \arrow[r, yshift=1.5ex,"s^{1}_{0}"]\arrow[r, yshift=-1.5ex,"t^{1}_{0}"]
 &f^{1}(1(1))
\arrow[r, yshift=1.5ex,"s^{2}_{1}"]\arrow[r, yshift=-1.5ex,"t^{2}_{1}"{below}]  
& f^{2}(1(2)) 
\arrow[r, yshift=1.5ex,"s^{3}_{2}"]\arrow[r, yshift=-1.5ex,"t^{3}_{2}"{below}]
&f^{3}(1(3)) 
\arrow[r, yshift=1.5ex,"s^{4}_{3}"]\arrow[r, yshift=-1.5ex,"t^{4}_{3}"{below}]
&f^{4}(1(4))\cdots f^{n-1}(1(n-1))  
\arrow[r, yshift=1.5ex,"s^{n}_{n-1}"]\arrow[r, yshift=-1.5ex,"t^{n}_{n-1}"{below}]
& f^{n}(1(n))\cdots\\ 
&\alpha_1(1(0))\arrow[r,"s^{2}_{1}"]
\arrow[rdd,"t^{2}_{1}",near start]&\alpha_2(1(0))\arrow[dash,dotted]{r}
&\alpha_{k-1}(1(0))\arrow[r,"s^{k}_{k-1}"]\arrow[rdd,"t^{k}_{k-1}",near start]&\alpha_{k}(1(0))\arrow[dash,dotted]{r}&
\alpha_{n-1}(1(0))\arrow[rd,"s^{n}_{n-1}"]\\
&&&&&&\xi_n(1(0))\\
&\beta_1(1(0))\arrow[r,"t^{2}_{1}"{below}]
\arrow[ruu,"s^{2}_{1}",near start]&
\beta_2(1(0))\arrow[dash,dotted]{r}
&\beta_{k-1}(1(0))\arrow[r,"t^{k}_{k-1}"{below}]\arrow[ruu,"s^{k}_{k-1}",near start]&\beta_{k}(1(0))\arrow[dash,dotted]{r}&
\beta_{n-1}(1(0))\arrow[ru,"t^{n}_{n-1}"{below}]\\
g^{0}(1(0))\arrow[ruuu,"s^{1}_{0}", near start]\arrow[ru,"t^{1}_{0}"]
 \arrow[r, yshift=1.5ex,"s^{1}_{0}"{below}]\arrow[r, yshift=-1.5ex,"t^{1}_{0}"{below}]
 &g^{1}(1(1))
\arrow[r, yshift=1.5ex,"s^{2}_{1}"]\arrow[r, yshift=-1.5ex,"t^{2}_{1}"{below}]  
& g^{2}(1(2)) 
\arrow[r, yshift=1.5ex,"s^{3}_{2}"]\arrow[r, yshift=-1.5ex,"t^{3}_{2}"{below}]
&g^{3}(1(3)) 
\arrow[r, yshift=1.5ex,"s^{4}_{3}"]\arrow[r, yshift=-1.5ex,"t^{4}_{3}"{below}]
&g^{4}(1(4))\cdots 
g^{n-1}(1(n-1))  
\arrow[r, yshift=1.5ex,"s^{n}_{n-1}"]\arrow[r, yshift=-1.5ex,"t^{n}_{n-1}"{below}]
& g^{n}(1(n))\cdots     
\end{tikzcd}$$

$$\begin{tikzcd}
 2(0)
 \arrow[rr, yshift=1.5ex,"s^{1}_{0}"]\arrow[rr, yshift=-1.5ex,"t^{1}_{0}"{below}]
 &&2(1)
\arrow[rr, yshift=1.5ex,"s^{2}_{1}"]\arrow[rr, yshift=-1.5ex,"t^{2}_{1}"{below}]  
&& 2(2) 
\arrow[rr, yshift=1.5ex,"s^{3}_{2}"]\arrow[rr, yshift=-1.5ex,"t^{3}_{2}"{below}]
&&2(3) 
\arrow[rr, yshift=1.5ex,"s^{4}_{3}"]\arrow[rr, yshift=-1.5ex,"t^{4}_{3}"{below}]
&&2(4)\cdots 2(n-1)  
\arrow[rr, yshift=1.5ex,"s^{n}_{n-1}"]\arrow[rr, yshift=-1.5ex,"t^{n}_{n-1}"{below}]
&& 2(n)\cdots    
\end{tikzcd}$$

With the small category $\mathbb{G}_{n}$ (for all $n\geq 1$) it is possible to associate 
kind of colimits that we call $\mathbb{G}_{n}$-\textit{sums} (see below), which are exactly objects 
 of the category $\big((1+2\downarrow\mathbb{S}^n(1+2)\big)$ (up to isomorphisms). 
These objects are also called $\mathbb{G}_{n}$-\textit{trees}.
Objects $a\in\mathbb{G}_n(0)$ have dimensions where notations of the diagrams above 
indicate these dimensions: $\text{dim}(1(n))=\text{dim}(2(n))=n$, but also 
$\text{dim}(f^{n}(1(n)))=\text{dim}(g^{n}(1(n)))=n$, $\text{dim}(\alpha_{k}(1(0)))=\text{dim}(\beta_{k}(1(0)))=k$, 
and also $\text{dim}(\xi_n(1(0)))=n$.
Thus objects $a\in\mathbb{G}_n(0)$ shall be denoted with a subscript $a_l$ where $l\in\mathbb{N}$ is the
dimension of $a_l$.
Objects of $\mathbb{G}_n$ are equipped with the natural order $<$ provided by arrows
of $\mathbb{G}_n$ : for $a, b\in\mathbb{G}_n(0)$ we have $a<b$ if and only if
there is an arrow  
\begin{tikzcd}
a\arrow[r]&b
\end{tikzcd}
in $\mathbb{G}_n$. We can extend this order $<$ with the order $\leq$ by saying:
$a\leq b$ if and only if $a<b$ or $a=b$.
Now a $\mathbb{G}_n$-tree is given by a table:

$$\begin{pmatrix} 
a_{i_1}&&&&a_{i_2}&&&&a_{i_3}&\cdot&\cdot&\cdot&a_{i_{k-1}}&&&&a_{i_{k}}\\\\
&&a_{i'_1}&&&&a_{i'_2}&&\cdot&\cdot&\cdot&\cdot&\cdot&&a_{i'_{k-1}}
\end{pmatrix}$$

such that $k\geq 1$, and for all $1\leq l\leq k$, $a_{i_{l}}$ are objects of $\mathbb{G}_n$ such that
$\text{dim}(a_{i_{l}})=i_l$, and for all $1\leq l\leq k-1$ we have $a_{i_l}\geq a_{i'_l}\leq a_{i_{l+1}}$.
It is straightforward to see that $\mathbb{G}_n$-trees are globular sets, and that for $n=0$ we 
recover $\mathbb{G}_0$-trees.

Let $\mathcal{C}$ a category and let 
\begin{tikzcd}
\mathbb{G}_n\arrow[rr,"F"]&&\mathcal{C}
\end{tikzcd}
a functor. We denote $F(a_{i_{k}})=D^{a_{i_{k}}}$ and we shall keep the same 
notations for the image of cosources : $F(s^{i_l}_{i_{l'}})=s^{i_l}_{i_{l'}}$, 
and for the image of cotargets : $F(t^{i_l}_{i_{l'}})=t^{i_l}_{i_{l'}}$, because
no risk of confusion will occur. In this case 
\begin{tikzcd}
\mathbb{G}_n\arrow[rr,"F"]&&\mathcal{C}
\end{tikzcd}
is called a globular $\mathbb{G}_n$-extension if for all $\mathbb{G}_n$-trees $t$ as just above, the
colimit 
of the following diagram exist in $\mathcal{C}$:

$$\begin{tikzcd}
D^{a_{i_1}}&&D^{a_{i_2}}&&\cdots&&D^{a_{i_{k-1}}}&&D^{a_{i_k}}\\
&D^{a_{i'_1}}\arrow[lu,"t^{i_1}_{i'_1}"]\arrow[ru,"s^{i_2}_{i'_1}"]
&&D^{a_{i'_2}}\arrow[lu,"t^{i_2}_{i'_2}"]&\cdots&
D^{a_{i'_{k-2}}}\arrow[ru]&&
D^{a_{i'_{k-1}}}\arrow[lu,"t^{i_{k-1}}_{i'_{k-1}}"]\arrow[ru,"s^{i_k}_{i'_{k-1}}"]
\end{tikzcd}$$

Following the terminology of Grothendieck, such colimits are called $\mathbb{G}_n$-\textit{globular sums}.

A morphism of globular $\mathbb{G}_n$-extensions, also called \textit{globular $\mathbb{G}_n$-functor}, is 
given by a commutative triangle in $\CAT$:

$$\begin{tikzcd}
&&\mathcal{C}\arrow[dd,"H"]\\
\mathbb{G}_n\arrow[rru,"F"]\arrow[rrd,"F'"{below}]\\
&&\mathcal{C}'
\end{tikzcd}$$

such that the functor $H$ preserves
$\mathbb{G}_n$-globular sums. The category of globular $\mathbb{G}_n$-extensions
is denoted $\mathbb{G}_n\text{-}\mathbb{E}\text{xt}$. In
fact this category has an initial object denoted 
\begin{tikzcd}
\mathbb{G}_n\arrow[rr,"i"]&&\Theta_n
\end{tikzcd}.
And the small category $\Theta_n$ can be described as the full subcategory
of $[\mathbb{G}^{op}+\mathbb{G}^{op};\E]$ which objects are globular $\mathbb{G}_n$-trees. 
In particular this small category $\Theta_n$ is the basic inductive sketch 
we shall need to describe coherators which set models are globular weak 
$(n,\infty)$-transformations ($n\in\mathbb{N}$) (see \ref{coherator-M}).

A globular $\mathbb{G}_n$-theory is given by a globular $\mathbb{G}_n$-extension
\begin{tikzcd}
\mathbb{G}_n\arrow[rr,"F"]&&\mathcal{C}
\end{tikzcd}
such that the unique induced functor $\overline{F}$ which makes
commutative the diagram:

$$\begin{tikzcd}
&&\Theta_n\arrow[dd,"\overline{F}"]\\
\mathbb{G}_n\arrow[rru,"i"]\arrow[rrd,"F"{below}]\\
&&\mathcal{C}
\end{tikzcd}$$

induces a bijection between objects of $\Theta_n$ and
objects of $\mathcal{C}$. The full subcategory of 
$\mathbb{G}_n\text{-}\mathbb{E}\text{xt}$ which objects
are globular $\mathbb{G}_n$-theories is denoted
$\mathbb{G}_n\text{-}\mathbb{T}\text{h}$.
Consider an object
\begin{tikzcd}
\mathbb{G}_n\arrow[rr,"F"]&&\mathcal{C}
\end{tikzcd}
of $\mathbb{G}_n\text{-}\mathbb{T}\text{h}$,
in particular it induces the globular $\mathbb{G}_n$-functor
\begin{tikzcd}
\Theta_n\arrow[rr,"\overline{F}"]&&\mathcal{C}
\end{tikzcd} 
as just above, which is a bijection on objects. A $\E$-model 
of $(F,\mathcal{C})$ or for $\mathcal{C}$ for short, is given 
by a functor :
\begin{tikzcd}
\mathcal{C}\arrow[rr,"X"]&&\E
\end{tikzcd},
such that the functor $X\circ\overline{F}$:

$$\begin{tikzcd}
\Theta_n\arrow[rr,"\overline{F}"]&&\mathcal{C}\arrow[rr,"X"]&&\E
\end{tikzcd}$$

sends globular $\mathbb{G}_n$-sums to globular $\mathbb{G}_n$-products\footnote{Globular $\mathbb{G}_n$-products
are just dual to globular $\mathbb{G}_n$-sums.}, thus for all objects $t$ of $\Theta_n$:

$$\begin{pmatrix} 
a_{i_1}&&&&a_{i_2}&&&&a_{i_3}&\cdot&\cdot&\cdot&a_{i_{k-1}}&&&&a_{i_{k}}\\\\
&&a_{i'_1}&&&&a_{i'_2}&&\cdot&\cdot&\cdot&\cdot&\cdot&&a_{i'_{k-1}}
\end{pmatrix}$$

we have:

\begin{eqnarray*}
X(\overline{F}(t)) & = &{X\left(\text{colim}\left(\begin{tikzcd}
D^{a_{i_1}}&&D^{a_{i_2}}&&\cdots&&D^{a_{i_{k-1}}}&&D^{a_{i_k}}\\
&D^{a_{i'_1}}\arrow[lu,"t^{i_1}_{i'_1}"]\arrow[ru,"s^{i_2}_{i'_1}"]
&&&\cdots&&&
D^{a_{i'_{k-1}}}\arrow[lu,"t^{i_{k-1}}_{i'_{k-1}}"]\arrow[ru,"s^{i_k}_{i'_{k-1}}"]
\end{tikzcd}\right)\right)} \\
 & = & {X\left( (D^{a_{i_1}},t^{i_1}_{i'_1})\underset{D^{a_{i'_1}}}\coprod (s^{i_2}_{i'_1},D^{a_{i_2}},t^{i_3}_{i'_2})
 \underset{D^{a_{i'_2}}}\coprod
\cdots\underset{D^{a_{i'_{k-1}}}}\coprod
(s^{i_k}_{i'_{k-1}},D^{a_{i_k}})\right)}\\
& \simeq &
{X(D^{a_{i_1}})\underset{X(D^{a_{i'_1}})}\times\cdots
\underset{X(D^{a_{i'_{k-1}}})}\times X(D^{a_{i_k}})}
\end{eqnarray*}

The category of $\E$-models of $\mathcal{C}$ is the full 
subcategory of the category of presheaves
$[\mathcal{C},\E]$ which objects are $\E$-models
of $\mathcal{C}$, and it is denoted $\mathbb{M}\text{od}(\mathcal{C})$.

Now it is easy to see that we get the following globular filtration: 

$$\begin{tikzcd}
 \mathbb{G}_0
 \arrow[dd]\arrow[rr, yshift=1.5ex,"s^{1}_{0}"]\arrow[rr, yshift=-1.5ex,"t^{1}_{0}"{below}]
 &&\mathbb{G}_1
\arrow[dd]\arrow[rr, yshift=1.5ex,"s^{2}_{1}"]\arrow[rr, yshift=-1.5ex,"t^{2}_{1}"{below}]  
&& \mathbb{G}_2 
\arrow[dd]\arrow[rr, yshift=1.5ex,"s^{3}_{2}"]\arrow[rr, yshift=-1.5ex,"t^{3}_{2}"{below}]
&&\mathbb{G}_3
\arrow[dd]\arrow[rr, yshift=1.5ex,"s^{4}_{3}"]\arrow[rr, yshift=-1.5ex,"t^{4}_{3}"{below}]
&&\mathbb{G}_4\arrow[dd,xshift=-4.7ex]\cdots \mathbb{G}_{n-1}  
\arrow[dd,xshift=1.6ex]\arrow[rr, yshift=1.5ex,"s^{n}_{n-1}"]\arrow[rr, yshift=-1.5ex,"t^{n}_{n-1}"{below}]
&& \mathbb{G}_{n}\arrow[dd,xshift=-2ex]\cdots\\\\
\Theta_0
 \arrow[rr, yshift=1.5ex,"s^{1}_{0}"]\arrow[rr, yshift=-1.5ex,"t^{1}_{0}"{below}]
 &&\Theta_1
\arrow[rr, yshift=1.5ex,"s^{2}_{1}"]\arrow[rr, yshift=-1.5ex,"t^{2}_{1}"{below}]  
&&\Theta_2 
\arrow[rr, yshift=1.5ex,"s^{3}_{2}"]\arrow[rr, yshift=-1.5ex,"t^{3}_{2}"{below}]
&&\Theta_3
\arrow[rr, yshift=1.5ex,"s^{4}_{3}"]\arrow[rr, yshift=-1.5ex,"t^{4}_{3}"{below}]
&&\Theta_4\cdots\Theta_{n-1}  
\arrow[rr, yshift=1.5ex,"s^{n}_{n-1}"]\arrow[rr, yshift=-1.5ex,"t^{n}_{n-1}"{below}]
&&\Theta_{n}\cdots    
\end{tikzcd}$$

\subsection{The theories $\Theta_{\mathbb{M}^n}$ for $(n,\infty)$-magmas ($n\in\mathbb{N}$) and their filtration.}
\label{4}

A $(0,\infty)$-magma $M$ is given by a globular set 
\begin{tikzcd}
\mathbb{G}^{\text{op}}\arrow[r,"M"]&\E
\end{tikzcd}
equipped with operations
\begin{tikzcd}
M_n\times_{M_p} M_n\arrow[r,"\circ^{n}_p"]&M_n
\end{tikzcd} 
for all $n\geq 1$ and all $0\leq p\leq n-1$ such that :
\begin{itemize}
\item for $0\leq p<q<m$, 
$s^m_q(y\circ^{m}_{p}x)=s^m_q(y)\circ^{q}_{p}s^m_q(x)$ and 
$t^m_q(y\circ^{m}_{p}x)=t^m_q(y)\circ^{q}_{p}t^m_q(x)$

\item for $0\leq q<p<m$, $s^m_q(y\circ^{m}_{p}x)=s^m_q(y)=s^m_q(x)$
and $t^m_q(y\circ^{m}_{p}x)=t^m_q(y)=t^m_q(x)$

\item for $0\leq p=q<m$, $s^m_q(y\circ^{m}_{p}x)=s^m_q(x)$ and 
$t^m_q(y\circ^{m}_{p}x)=t^m_q(x)$

\end{itemize}

A $(1,\infty)$-magma
\begin{tikzcd}
M\arrow[r,"F"]&M'
\end{tikzcd}
between two $(0,\infty)$-magma $M$ and $M'$ is a morphism of globular sets. 
If 
\begin{tikzcd}
N\arrow[r,"G"]&N'
\end{tikzcd}
is another $(1,\infty)$-magma, then a strict morphism between them:
\begin{tikzcd}
F\arrow[r,"{(h,h')}"]&G
\end{tikzcd}
is given by two morphisms of $(0,\infty)$-magma:
\begin{tikzcd}
M\arrow[r,"h"]&N
\end{tikzcd}
and 
\begin{tikzcd}
M'\arrow[r,"h'"]&N'
\end{tikzcd}
such that the following diagram is commutative in $\Glob$:

$$\begin{tikzcd}
M\arrow[dd,"h"{left}]\arrow[rr,"F"]&&M'\arrow[dd,"h'"]\\\\
N\arrow[rr,"G"{below}]&&N'
\end{tikzcd}$$

The category of $(1,\infty)$-magmas is denoted $(1,\infty)\text{-}\mathbb{M}\text{ag}$.

A $(2,\infty)$-magma
\begin{tikzcd}
 M\arrow[r, bend  left=50, "F"{name=U}]
  \arrow[r, bend right=50, "G"{name=D, below}]
& M'
\arrow[Rightarrow,"\tau"{left}, from=U, to=D]
\end{tikzcd}
is given by two objects 
\begin{tikzcd}
M\arrow[r,"F"]&M'
\end{tikzcd}
and 
\begin{tikzcd}
M\arrow[r,"G"]&M'
\end{tikzcd}
of $(1,\infty)\text{-}\mathbb{M}\text{ag}$
and a morphism:
\begin{tikzcd}
M(0)\arrow[r,"\tau"]&M'(1)
\end{tikzcd}
of $\E$ such that for all $a\in M(0)$, $s^{1}_{0}(\tau(a))=F(a)$ and $t^{1}_{0}(\tau(a))=G(a)$.
If 
\begin{tikzcd}
 N\arrow[r, bend  left=50, "H"{name=U}]
  \arrow[r, bend right=50, "K"{name=D, below}]
& N'
\arrow[Rightarrow,"\rho"{left}, from=U, to=D]
\end{tikzcd}
is another $(2,\infty)$-magma, then a strict morphism between them:
\begin{tikzcd}
\tau\arrow[r,"{(h,h')}"]&\rho
\end{tikzcd}
is given by two morphisms of $(0,\infty)$-magmas:
\begin{tikzcd}
M\arrow[r,"h"]&N
\end{tikzcd}
and 
\begin{tikzcd}
M'\arrow[r,"h'"]&N'
\end{tikzcd}
such that for all $a\in M(0)$ we have: $h'(\tau(a))=\rho(h(a))$.
The category of $(2,\infty)$-magmas is denoted $(2,\infty)\text{-}\mathbb{M}\text{ag}$.

For $n\geq 3$ we define $(n,\infty)$-magmas by induction: suppose the categories 
$(k,\infty)\text{-}\mathbb{M}\text{ag}$ of $(k,\infty)$-magmas are defined for all 
$k\in\llbracket 2,n-1\rrbracket$. An $(n,\infty)$-magma
\begin{tikzcd}
\alpha\arrow[r,"\xi"]&\beta
\end{tikzcd}
between the $(n-1,\infty)$-magmas $\alpha$ and $\beta$ is given by a 
morphism 
\begin{tikzcd}
M(0)\arrow[r,"\xi"]&M'(n)
\end{tikzcd}
in $\E$ such that for all $a\in M(0)$, $s^{n}_{0}(\xi(a))=F(a)$ and $t^{n}_{0}(\xi(a))=G(a)$. 
A morphism between two $(n,\infty)$-magmas
\begin{tikzcd}
\xi\arrow[r,"{(h,h')}"]&\xi'
\end{tikzcd} 
is given by two morphisms of $(0,\infty)$-magmas:
\begin{tikzcd}
M\arrow[r,"h"]&N
\end{tikzcd}
and 
\begin{tikzcd}
M'\arrow[r,"h'"]&N'
\end{tikzcd}
such that for all $a\in M(0)$ we have: $h'(\xi(a))=\xi'(h(a))$.
The category of $(n,\infty)$-magmas is denoted $(n,\infty)\text{-}\mathbb{M}\text{ag}$.

Now for all $n\geq 1$ we have a monadic forgetful functor $U^n$:

$$\begin{tikzcd}
(n,\infty)\text{-}\mathbb{M}\text{ag}\arrow[dd,xshift=1ex,"\dashv"{left},"U^n"]\\\\
\mathbb{G}\text{lob}^2\arrow[uu,xshift=-1ex,"F^n",dotted] 
  \end{tikzcd}$$
  
from the category $(n,\infty)\text{-}\mathbb{M}\text{ag}$ of globular $(n,\infty)$-magmas to the category $\mathbb{G}\text{lob}^2$ of 
pairs of globular sets, which sends $(n,\infty)$-magmas $\xi$ as above to the pair of globular sets $(U(M),U(M'))$. The notation $U(M)$ 
means that we have forgotten the underlying structure of $(0,\infty)$-magma of $M$. We use the equivalence 
$\mathbb{G}\text{lob}^2\simeq [\mathbb{G}_0^{op}+\mathbb{G}_0^{op};\E]$ to get a monad 
$\mathbb{M}^n=(M^n,\eta^n,\mu^n)$ on $[\mathbb{G}_0^{op}+\mathbb{G}_0^{op};\E]$. We have the 
equivalence of categories 
$(n,\infty)\text{-}\mathbb{M}\text{ag}\simeq
\mathbb{M}^n\text{-}\mathbb{A}\text{lg}$ because $U^n$ is 
monadic. The full subcategory 
$\Theta_{\mathbb{M}^n}\subset\mathbb{K}\text{l}(\mathbb{M}^n)$ of
the Kleisli category of $\mathbb{M}^n$ which objects are $\mathbb{G}_n$-trees is
called the theory of globular $(n,\infty)$-magmas. This is our basic example 
of globular $\mathbb{G}_n$-theory. In fact we have the following equivalences of categories:

$$(n,\infty)\text{-}\mathbb{M}\text{ag}\simeq
\mathbb{M}^n\text{-}\mathbb{A}\text{lg}\simeq
\mathbb{M}\text{od}(\Theta_{\mathbb{M}^n})$$

Now it is easy to see that we get the following globular filtration: 

$$\begin{tikzcd}
 \mathbb{G}_0
 \arrow[dd]\arrow[rr, yshift=1.5ex,"s^{1}_{0}"]\arrow[rr, yshift=-1.5ex,"t^{1}_{0}"{below}]
 &&\mathbb{G}_1
\arrow[dd]\arrow[rr, yshift=1.5ex,"s^{2}_{1}"]\arrow[rr, yshift=-1.5ex,"t^{2}_{1}"{below}]  
&& \mathbb{G}_2 
\arrow[dd]\arrow[rr, yshift=1.5ex,"s^{3}_{2}"]\arrow[rr, yshift=-1.5ex,"t^{3}_{2}"{below}]
&&\mathbb{G}_3
\arrow[dd]\arrow[rr, yshift=1.5ex,"s^{4}_{3}"]\arrow[rr, yshift=-1.5ex,"t^{4}_{3}"{below}]
&&\mathbb{G}_4\arrow[dd,xshift=-4.7ex]\cdots\mathbb{G}_{n-1}  
\arrow[dd,xshift=1.6ex]\arrow[rr, yshift=1.5ex,"s^{n}_{n-1}"]\arrow[rr, yshift=-1.5ex,"t^{n}_{n-1}"{below}]
&& \mathbb{G}_{n}\arrow[dd,xshift=-2ex]\cdots\\\\
\Theta_0
 \arrow[dd]\arrow[rr, yshift=1.5ex,"s^{1}_{0}"]\arrow[rr, yshift=-1.5ex,"t^{1}_{0}"{below}]
 &&\Theta_1
\arrow[dd]\arrow[rr, yshift=1.5ex,"s^{2}_{1}"]\arrow[rr, yshift=-1.5ex,"t^{2}_{1}"{below}]  
&&\Theta_2 
\arrow[dd]\arrow[rr, yshift=1.5ex,"s^{3}_{2}"]\arrow[rr, yshift=-1.5ex,"t^{3}_{2}"{below}]
&&\Theta_3
\arrow[dd]\arrow[rr, yshift=1.5ex,"s^{4}_{3}"]\arrow[rr, yshift=-1.5ex,"t^{4}_{3}"{below}]
&&\Theta_4\arrow[dd,xshift=-4.7ex]\cdots\Theta_{n-1}  
\arrow[dd,xshift=1.6ex]\arrow[rr, yshift=1.5ex,"s^{n}_{n-1}"]\arrow[rr, yshift=-1.5ex,"t^{n}_{n-1}"{below}]
&&\Theta_{n}\arrow[dd,xshift=-2ex]\cdots\\\\
 \Theta_{\mathbb{M}^0}
  \arrow[rr, yshift=1.5ex,"s^{1}_{0}"]\arrow[rr, yshift=-1.5ex,"t^{1}_{0}"{below}]
 &&\Theta_{\mathbb{M}^1}
 \arrow[rr, yshift=1.5ex,"s^{2}_{1}"]\arrow[rr, yshift=-1.5ex,"t^{2}_{1}"{below}]  
&&\Theta_{\mathbb{M}^2}
\arrow[rr, yshift=1.5ex,"s^{3}_{2}"]\arrow[rr, yshift=-1.5ex,"t^{3}_{2}"{below}]
&&\Theta_{\mathbb{M}^3}
\arrow[rr, yshift=1.5ex,"s^{4}_{3}"]\arrow[rr, yshift=-1.5ex,"t^{4}_{3}"{below}]
&&\Theta_{\mathbb{M}^4}\cdots\Theta_{\mathbb{M}^{n-1}}
\arrow[rr, yshift=1.5ex,"s^{n}_{n-1}"]\arrow[rr, yshift=-1.5ex,"t^{n}_{n-1}"{below}]
&&\Theta_{\mathbb{M}^n}\cdots  
\end{tikzcd}$$

\section{Globular $\mathbb{G}_n$-coherators ($n\in\mathbb{N}$)}

\subsection{Admissibility}

For a fixed integer $n\geq 1$, let 
\begin{tikzcd}
\mathbb{G}_n\arrow[rr,"F"]&&\mathcal{C}
\end{tikzcd}
be a globular $\mathbb{G}_n$-theory, i.e an object of 
$\mathbb{G}_n\text{-}\mathbb{T}\text{h}$; here we pay 
attention to the inclusion:
\begin{tikzcd}
\mathbb{G}_0+\mathbb{G}_0\arrow[rr,hook]&&\mathbb{G}_n
\end{tikzcd}. An object $c_m$ of $\mathbb{G}_0+\mathbb{G}_0$
can be the object $1(m)$ or the object $2(m)$ for all $m\in\mathbb{N}$;
two arrows:
\begin{tikzcd}
D^{c_m}\arrow[rr, yshift=1.2ex,"f"]
 \arrow[rr, yshift=-1.2ex,"g"{below}]&&t
\end{tikzcd}
in $\mathcal{C}$ are parallels if $fs^{m}_{m-1}=gs^{m}_{m-1}$
and $ft^{m}_{m-1}=gt^{m}_{m-1}$:
$$\begin{tikzcd}
D^{c_m}\arrow[rr, yshift=1.2ex,"f"]
 \arrow[rr, yshift=-1.2ex,"g"{below}]&&t\\\\
 D^{c_{m-1}}\arrow[uu, xshift=-1.2ex,"s^{m}_{m-1}"{left}]
 \arrow[uu, xshift=1.2ex,"t^{m}_{m-1}"{right}] 
\end{tikzcd}$$

Consider a couple $(f,g)$ of parallels arrows in $\mathcal{C}$
as just above. We say that it is admissible or algebraic if they
don't belong to the image of the globular $\mathbb{G}_n$-functor $\overline{F}$ :
$$\begin{tikzcd}
&&\Theta_n\arrow[dd,"\overline{F}"]\\
\mathbb{G}_n\arrow[rru,"i"]\arrow[rrd,"F"{below}]\\
&&\mathcal{C}
\end{tikzcd}$$

Consider a couple $(f,g)$ of arrows of $\mathcal{C}$ which is admissible
 as just above; a lifting of $(f,g)$ is given by an arrow $h$ :
$$\begin{tikzcd}
D^{c_{m+1}}\arrow[rrdd,"h"]\\\\
D^{c_m}\arrow[uu, xshift=-1.2ex,"s^{m+1}_{m}"{left}]
 \arrow[uu, xshift=1.2ex,"t^{m+1}_{m}"{right}] 
\arrow[rr, yshift=1.2ex,"f"]
 \arrow[rr, yshift=-1.2ex,"g"{below}]&&t
 \end{tikzcd}$$
 such that $hs^{m+1}_{m}=f$ and $ht^{m+1}_{m}=g$
 
 \subsection{Batanin-Grothendieck Sequences}
 
 We now define the Batanin-Grothendieck sequence\footnote{Coherators associated to such sequence
 are called \textit{of Batanin-Leinster type} by some authors.} associated to a globular $\mathbb{G}_n$-theory 
 \begin{tikzcd}
\mathbb{G}_n\arrow[rr,"F"]&&\mathcal{C}
\end{tikzcd}. 
We build it by the following induction:

\begin{itemize}
\item If $m=0$ we start with the couple $(\mathcal{C},E)$ where 
$E$ denotes the set of admissible pairs of arrows of $\mathcal{C}$;
we shall write $(\mathcal{C}_0,E_0)=(\mathcal{C},E)$ this first
step.

\item If $m=1$ we consider then the couple $(\mathcal{C}_1,E_1)$
where $\mathcal{C}_1$ is obtained by formally adding in 
$\mathcal{C}_0=\mathcal{C}$ the liftings of all elements 
$(f,g)\in E_0=E$, and $E_1$ is the set of admissible couples of arrows 
in $\mathcal{C}_1$ which are not elements of the set $E_0$;

\item If for $m\geq 2$ the couple $(\mathcal{C}_m,E_m)$ is well defined
then $\mathcal{C}_{m+1}$ is obtained by formally adding in 
$\mathcal{C}_{m}$ the liftings of all elements of $E_m$, and 
$E_{m+1}$ is the set of couples of arrows of $\mathcal{C}_{m+1}$
which are not elements of $E_m$
\end{itemize}
We give a slightly different but equivalent induction to build the Batanin-Grothendieck sequence for such globular theory
\begin{tikzcd}
\mathbb{G}_n\arrow[rr,"F"]&&\mathcal{C}
\end{tikzcd}:

\begin{itemize}
\item If $m=0$ we start with the couple $(\mathcal{C},E)$ where 
$E$ is the set of couple of arrow which are admissible of $\mathcal{C}$;
we denote $E=E_0=E'_0=E'_0\setminus\emptyset$ (we shall see soon the reason of these notations), and $\mathcal{C}_0=\mathcal{C}$;

\item If $m=1$ we consider the couple $(\mathcal{C}_1,E_1)$ where
$\mathcal{C}_1$ is obtained by formally adding in $\mathcal{C}_0$
all liftings of the elements $(f,g)\in E_0$, $E'_1$ is the set of all
pairs of arrows which are admissible in $\mathcal{C}_1$, and 
$E_1=E'_1\setminus E_0$; remark that $E_0=E'_0\subset E'_1$;

\item If $m=2$ we consider the couple $(\mathcal{C}_2,E_2)$ where
$\mathcal{C}_2$ is obtained by formally adding in $\mathcal{C}_1$
all liftings of the elements $(f,g)\in E_1$, $E'_2$ is the set of all
pairs of arrows which are admissible in $\mathcal{C}_2$, and
$E_2=E'_2\setminus E'_1$;

\item For $m\geq 3$ we suppose that the couple $(\mathcal{C}_m,E_m)$
is well defined with $E_m=E'_m\setminus E'_{m-1}$, then 
$\mathcal{C}_{m+1}$ is obtained by formally adding in $\mathcal{C}_m$
all liftings of the elements $(f,g)\in E_m$, $E'_{m+1}$ is the set of all
pairs of arrows which are admissible in $\mathcal{C}_{m+1}$, and
$E_{m+1}=E'_{m+1}\setminus E'_m$;
\end{itemize}

 The Batanin-Grothendieck sequence of the globular theory
\begin{tikzcd}
\mathbb{G}_n\arrow[rr,"F"]&&\mathcal{C}
\end{tikzcd} 
produces the following filtered diagram:

\begin{tikzcd}
(\mathbb{N},\leq)\arrow[rr,"\mathcal{C}_{\bullet}"]&&
\mathbb{G}_n\text{-}\mathbb{T}\text{h}
\end{tikzcd}
in the category $\mathbb{G}_n\text{-}\mathbb{T}\text{h}$:

$$\begin{tikzcd}
\mathcal{C}_0\arrow[rr,"i_1"]&&\mathcal{C}_1\arrow[rr,"i_2"]&&\cdots\arrow[rr,"i_m"]&&\mathcal{C}_m\arrow[rr]&&\cdots
\end{tikzcd}$$

\subsection{$\mathbb{G}_n$-coherators for globular $\mathbb{G}_n$-theories ($n\in\mathbb{N}$)}

Let us fixed an integer $n\geq 1$. We start with datas of the previous subsection, i.e 
with the Batanin-Grothendieck sequence \begin{tikzcd}
(\mathbb{N},\leq)\arrow[rr,"\mathcal{C}_{\bullet}"]&&
\mathbb{G}_n\text{-}\mathbb{T}\text{h}
\end{tikzcd}
for a globular $\mathbb{G}_n$-theory 
\begin{tikzcd}
\mathbb{G}_n\arrow[rr,"F"]&&\mathcal{C}
\end{tikzcd}:

\begin{definition}
 The colimit 
\begin{tikzcd}
\mathbb{G}_n\arrow[rr,"F_{\infty}"]&&\mathcal{C}_{\infty}
\end{tikzcd}
of the previous filtered diagram $\mathcal{C}_{\bullet}$ :
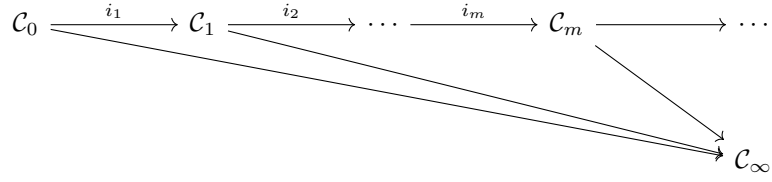
$$\begin{tikzcd}
\mathcal{C}_0\arrow[rr,"i_1"]\arrow[rrrrrrrrdd]
&&\mathcal{C}_1\arrow[rr,"i_2"]\arrow[rrrrrrdd]&&\cdots\arrow[rr,"i_m"]&&\mathcal{C}_m\arrow[rr]\arrow[rrdd]&&\cdots\\\\
&&&&&&&&\mathcal{C}_{\infty}
\end{tikzcd}$$
is called the globular $\mathbb{G}_n$-coherator of the type Batanin-Grothendieck associated to the globular $\mathbb{G}_n$-theory
\begin{tikzcd}
\mathbb{G}_n\arrow[rr,"F"]&&\mathcal{C}
\end{tikzcd}
\end{definition}
For shorter terminology we shall say that 
\begin{tikzcd}
\mathbb{G}_n\arrow[rr,"F_{\infty}"]&&\mathcal{C}_{\infty}
\end{tikzcd}
is the $\mathbb{G}_n$-coherator associated to the globular $\mathbb{G}_n$-theory
\begin{tikzcd}
\mathbb{G}_n\arrow[rr,"F"]&&\mathcal{C}
\end{tikzcd}.
It is straightforward to see that the Batanin-Grothendieck construction of $\mathbb{G}_n$-coherators associated to globular $\mathbb{G}_n$-theory 
is functorial, and the following functor $\Phi_n$ is called the Batanin-Grothendieck functor:

$$\begin{tikzcd}
\mathbb{G}_n\text{-}\mathbb{T}\text{h}
\arrow[rr,"\Phi_n"]&&\mathbb{G}_n\text{-}\mathbb{T}\text{h}\\
\mathcal{C}\arrow[rr,mapsto]&&\mathcal{C}^{\infty}
\end{tikzcd}$$

\subsection{The $\mathbb{G}_n$-coherators $\Theta^{\infty}_{\mathbb{M}^n}$ ($n\in\mathbb{N}$)}
\label{coherator-M}

Let us fixed an integer $n\geq 1$. The $\mathbb{G}_n$-coherator associated to the globular theory 
\begin{tikzcd}
\mathbb{G}_n\arrow[rr,"j"]&&\Theta_{\mathbb{M}^n}
\end{tikzcd}
that we obtain with the composition:

$$\begin{tikzcd}
\mathbb{G}_n\arrow[rr,"i"]&&\Theta_{n}
\arrow[rr,hook]&&\Theta_{\mathbb{M}^n}
\end{tikzcd}$$

is denoted $\Theta^{\infty}_{\mathbb{M}^n}$.

\begin{definition}
Globular weak $(n,\infty)$-transformations are objects of the category $\mathbb{M}\text{od}(\Theta^{\infty}_{\mathbb{M}^n})$.
\end{definition}

Morphisms in $\Theta^{\infty}_{\mathbb{M}^n}$ which doesn't belong to the image of
\begin{tikzcd}
\Theta_{n}
\arrow[rr,hook]&&\Theta^{\infty}_{\mathbb{M}^n}
\end{tikzcd}
are called \textit{algebraic morphisms} of $\Theta^{\infty}_{\mathbb{M}^n}$. Morphisms in $\Theta^{\infty}_{\mathbb{M}^n}$
of the form:
\begin{tikzcd}
1(p)\arrow[r,"K"]&t
\end{tikzcd}
or 
\begin{tikzcd}
2(p)\arrow[r,"K"]&t
\end{tikzcd},
where $t$ are any $\mathbb{G}_n$-trees, 
are the $p$-\textit{operations} of $\Theta^{\infty}_{\mathbb{M}^n}$.

\begin{definition}
If $\tau\in\mathbb{M}\text{od}(\Theta^{\infty}_{\mathbb{M}^n})$ and $p\geq 1$ is an integer, then $\text{dim}(\tau)=p$ if 
for all $q>p$, any $q$-operations $K$ in $\Theta^{\infty}_{\mathbb{M}^n}$ has its source and target which are equalize 
by $\tau$. Thus if we have:

$$\begin{tikzcd}
1(q)\arrow[rrrdd,"K"]\\\\
1(q-1)\arrow[uu,xshift=1.1ex,"t^{q}_{q-1}"{right}]\arrow[uu,xshift=-1.1ex,"s^{q}_{q-1}"]
\arrow[rrr, yshift=1.1ex,"{K\circ s^{q}_{q-1}}"]
\arrow[rrr, yshift=-1.1ex,"{K\circ t^{q}_{q-1}}"{below}]&&&t
\end{tikzcd}$$

or:

$$\begin{tikzcd}
2(q)\arrow[rrrdd,"K"]\\\\
2(q-1)\arrow[uu,xshift=1.1ex,"t^{q}_{q-1}"{right}]\arrow[uu,xshift=-1.1ex,"s^{q}_{q-1}"]
\arrow[rrr, yshift=1.1ex,"{K\circ s^{q}_{q-1}}"]
\arrow[rrr, yshift=-1.1ex,"{K\circ t^{q}_{q-1}}"{below}]&&&t
\end{tikzcd},$$

then $\tau(K\circ s^{q}_{q-1})=\tau(K\circ t^{q}_{q-1})$.
\end{definition}

\begin{remark}
In 2019 John Bourke has proved \cite{bourke-injectif} the \textit{Ara conjecture} \cite{ara-these} which says that the category of 
globular weak $\infty$-categories of Batanin is equivalent to the category of globular weak 
$\infty$-categories of Grothendieck:

$$\mathbb{M}\text{od}(\Theta_{\mathbb{B}^{0}_C})
\simeq\mathbb{M}\text{od}(\Theta^{\infty}_{\mathbb{M}})$$

where here $\mathbb{B}^{0}_C$ denotes the globular operad of Batanin \cite{batanin-main} which 
algebras are his models of globular weak $\infty$-categories and $\Theta_{\mathbb{B}^{0}_C}$ is its associated theory. In the 
same veine, for all integers $n\geq 1$, the category of 
globular weak $(n,\infty)$-transformations defined in \cite{cam-cgasa3} should be equivalent to the category of globular weak 
$(n,\infty)$-transformations with this $\mathbb{G}_n$-coherator $\Theta^{\infty}_{\mathbb{M}^m}$:

$$\mathbb{M}\text{od}(\Theta_{\mathbb{B}^{n}_C})
\simeq\mathbb{M}\text{od}(\Theta^{\infty}_{\mathbb{M}^n})$$

where $\mathbb{B}^{n}_C$ denotes the globular operad \cite{cam-cgasa3} which 
algebras are models of globular weak $(n,\infty)$-transformations and $\Theta_{\mathbb{B}^{n}_C}$ is its associated theory.
\end{remark}

Now it is easy to see that we get the following globular filtration: 

$$\begin{tikzcd}
 \mathbb{G}_0
 \arrow[dd]\arrow[rr, yshift=1.5ex,"s^{1}_{0}"]\arrow[rr, yshift=-1.5ex,"t^{1}_{0}"{below}]
 &&\mathbb{G}_1
\arrow[dd]\arrow[rr, yshift=1.5ex,"s^{2}_{1}"]\arrow[rr, yshift=-1.5ex,"t^{2}_{1}"{below}]  
&& \mathbb{G}_2 
\arrow[dd]\arrow[rr, yshift=1.5ex,"s^{3}_{2}"]\arrow[rr, yshift=-1.5ex,"t^{3}_{2}"{below}]
&&\mathbb{G}_3
\arrow[dd]\arrow[rr, yshift=1.5ex,"s^{4}_{3}"]\arrow[rr, yshift=-1.5ex,"t^{4}_{3}"{below}]
&&\mathbb{G}_4\arrow[dd,xshift=-4.7ex]\cdots\mathbb{G}_{n-1}  
\arrow[dd,xshift=1.6ex]\arrow[rr, yshift=1.5ex,"s^{n}_{n-1}"]\arrow[rr, yshift=-1.5ex,"t^{n}_{n-1}"{below}]
&& \mathbb{G}_{n}\arrow[dd,xshift=-2ex]\cdots\\\\
\Theta_0
 \arrow[dd]\arrow[rr, yshift=1.5ex,"s^{1}_{0}"]\arrow[rr, yshift=-1.5ex,"t^{1}_{0}"{below}]
 &&\Theta_1
\arrow[dd]\arrow[rr, yshift=1.5ex,"s^{2}_{1}"]\arrow[rr, yshift=-1.5ex,"t^{2}_{1}"{below}]  
&&\Theta_2 
\arrow[dd]\arrow[rr, yshift=1.5ex,"s^{3}_{2}"]\arrow[rr, yshift=-1.5ex,"t^{3}_{2}"{below}]
&&\Theta_3
\arrow[dd]\arrow[rr, yshift=1.5ex,"s^{4}_{3}"]\arrow[rr, yshift=-1.5ex,"t^{4}_{3}"{below}]
&&\Theta_4\arrow[dd,xshift=-4.7ex]\cdots\Theta_{n-1}  
\arrow[dd,xshift=1.6ex]\arrow[rr, yshift=1.5ex,"s^{n}_{n-1}"]\arrow[rr, yshift=-1.5ex,"t^{n}_{n-1}"{below}]
&&\Theta_{n}\arrow[dd,xshift=-2ex]\cdots\\\\
 \Theta_{\mathbb{M}^0}
 \arrow[dd] \arrow[rr, yshift=1.5ex,"s^{1}_{0}"]\arrow[rr, yshift=-1.5ex,"t^{1}_{0}"{below}]
 &&\Theta_{\mathbb{M}^1}
 \arrow[dd]\arrow[rr, yshift=1.5ex,"s^{2}_{1}"]\arrow[rr, yshift=-1.5ex,"t^{2}_{1}"{below}]  
&&\Theta_{\mathbb{M}^2}
\arrow[dd]\arrow[rr, yshift=1.5ex,"s^{3}_{2}"]\arrow[rr, yshift=-1.5ex,"t^{3}_{2}"{below}]
&&\Theta_{\mathbb{M}^3}
\arrow[dd]\arrow[rr, yshift=1.5ex,"s^{4}_{3}"]\arrow[rr, yshift=-1.5ex,"t^{4}_{3}"{below}]
&&\Theta_{\mathbb{M}^4}\arrow[dd,xshift=-4.7ex]\cdots\Theta_{\mathbb{M}^{n-1}}\arrow[dd]
\arrow[dd,xshift=1.6ex]\arrow[rr, yshift=1.5ex,"s^{n}_{n-1}"]\arrow[rr, yshift=-1.5ex,"t^{n}_{n-1}"{below}]
&&\Theta_{\mathbb{M}^n}\arrow[dd,xshift=-2ex]\cdots \\\\
\Theta^{\infty}_{\mathbb{M}^0}
  \arrow[rr, yshift=1.5ex,"s^{1}_{0}"]\arrow[rr, yshift=-1.5ex,"t^{1}_{0}"{below}]
 &&\Theta^{\infty}_{\mathbb{M}^1}
 \arrow[rr, yshift=1.5ex,"s^{2}_{1}"]\arrow[rr, yshift=-1.5ex,"t^{2}_{1}"{below}]  
&&\Theta^{\infty}_{\mathbb{M}^2}
\arrow[rr, yshift=1.5ex,"s^{3}_{2}"]\arrow[rr, yshift=-1.5ex,"t^{3}_{2}"{below}]
&&\Theta^{\infty}_{\mathbb{M}^3}
\arrow[rr, yshift=1.5ex,"s^{4}_{3}"]\arrow[rr, yshift=-1.5ex,"t^{4}_{3}"{below}]
&&\Theta^{\infty}_{\mathbb{M}^4}\cdots\Theta^{\infty}_{\mathbb{M}^{n-1}}
\arrow[rr, yshift=1.5ex,"s^{n}_{n-1}"]\arrow[rr, yshift=-1.5ex,"t^{n}_{n-1}"{below}]
&&\Theta^{\infty}_{\mathbb{M}^n}\cdots 
\end{tikzcd}$$

\subsection{Computations in dimensions $2$}

A natural transformation $\tau$:

$$\begin{tikzcd}
X\arrow[r, bend left=50, ""{name=F}{above},"F"]
\arrow[r, bend right=50, ""{name=H}{below},"H"{below}]
&Z
\arrow[Rightarrow,"\tau"{left}, from=F,to=H]
\end{tikzcd}$$

where $X$ and $Z$ are categories and $F$ and $H$ are functors, is given 
by a class of arrows:

$$\left(\begin{tikzcd}
F(x)\arrow[r,"{\tau(x)}"]&H(x)
\end{tikzcd}\right)_{x\in X(0)}$$

in $Z$ such that if 
\begin{tikzcd}
x\arrow[r,"a"]&y
\end{tikzcd}
is an arrow of $X$, then we have the following commutative diagrams:

$$\begin{tikzcd}
F(x)\arrow[dd,"F(a)"{left}]\arrow[rr,"{\tau(x)}"]&&H(x)\arrow[dd,"H(a)"]\\\\
F(y)\arrow[rr,"{\tau(y)}"{below}]&&H(y)
\end{tikzcd}$$

If instead the datas:

$$\begin{tikzcd}
X\arrow[r, bend left=50, ""{name=F}{above},"F"]
\arrow[r, bend right=50, ""{name=H}{below},"H"{below}]
&Z
\arrow[Rightarrow,"\tau"{left}, from=F,to=H]
\end{tikzcd}$$

are given by bicategories $X$ and $Z$, pseudo-$2$-functors $F$ and $H$, 
then a class of $1$-cells:

$$\left(\begin{tikzcd}
F(x)\arrow[r,"{\tau(x)}"]&H(x)
\end{tikzcd}\right)_{x\in X(0)}$$

of the bicategory $Z$, and if 
\begin{tikzcd}
x\arrow[r,"a"]&y
\end{tikzcd}
is a $1$-cell of $X$, then we have the following coherence $2$-cell $\omega(a)$:

$$\begin{tikzcd}
  F(x)\arrow[rr,"{\tau(x)}"] \ar[dd,"F(a)"{left}]
  \arrow[from=rr,to=dd,phantom,"\omega(a)"{above},""{name=1,near start},""{name=2,near end}]
  \arrow[Rightarrow,from=1,to=2]&&H(x)\arrow[dd,"H(a)"]\\\\
  F(y)\arrow[rr,"{\tau(y)}"{below}]&&H(y)
\end{tikzcd}$$

such that if 
\begin{tikzcd}
x\arrow[r,"a"]&y\arrow[r,"b"]&z
\end{tikzcd}
are $1$-cells of $X$ then we have the following commutative diagram:

$$\begin{tikzcd}
H_1(b)\circ^{1}_{0}(H_1(a)\circ^{1}_{0}\tau_1(x))\arrow[d,"{a(H_1(b),H_1(a),\tau_1(x))}"{left}]\arrow[rr,"{1_{H_1(b)}\circ^{2}_{0}\omega(a)}"]
&&H_1(b)\circ^{1}_{0}(\tau_1(y)\circ^{1}_{0}F_1(a))\arrow[d,,"{a(H_1(b),\tau_1(y),F_1(a))}"]\\
(H_1(b)\circ^{1}_{0}H_1(a))\circ^{1}_{0}\tau_1(x)\arrow[d,"{d_1(a,b)\circ^{2}_{0}1_{\tau_1(x)}}"{left}]
&&(H_1(b)\circ^{1}_{0}\tau_1(y))\circ^{1}_{0}F_1(a)\arrow[d,"{\omega(b)\circ{2}_{0}1_{F_1(a)}}"]\\
H_1(b\circ^{1}_{0}a)\circ^{1}_{0}\tau_1(x)\arrow[d,"{\omega(b\circ{1}_{0}a)}"{left}]
&&(\tau_1(z)\circ^{1}_{0}F_1(b))\circ^{1}_{0}F_1(a)\arrow[d,"{a(\tau_1(z),F_1(b),F_1(a))}"]\\
\tau_1(z)\circ^{1}_{0}F_1(b\circ^{1}_{0}a)
&&\tau_1(z)\circ^{1}_{0}(F_1(b)\circ^{1}_{0}F_1(a))\arrow[ll,"{1_{\tau_1(z)\circ^{2}_{0}d_0(b,a)}}"{below}]
\end{tikzcd}$$

\begin{definition}
Such $\tau$ described above are called pseudo-$2$-natural transformations. 
\end{definition}

We are going to show that $\Theta^{\infty}_{\mathbb{M}^2}$-models in $\E$ of dimension $2$, i.e globular weak $(2,\infty)$-natural 
transformations of dimension $2$, are pseudo-$2$-natural transformations. 

The following diagram of theories:

$$\begin{tikzcd}
\Theta^{\infty}_{\mathbb{M}^0}
  \arrow[rr, yshift=1.5ex,"s^{1}_{0}"]\arrow[rr, yshift=-1.5ex,"t^{1}_{0}"{below}]
 &&\Theta^{\infty}_{\mathbb{M}^1}
 \arrow[rr, yshift=1.5ex,"s^{2}_{1}"]\arrow[rr, yshift=-1.5ex,"t^{2}_{1}"{below}]  
&&\Theta^{\infty}_{\mathbb{M}^2}
\end{tikzcd}$$

leads, by passing to $\E$-models, to the following diagram in $\CAT$:

$$\begin{tikzcd}
\mathbb{M}\text{od}(\Theta^{\infty}_{\mathbb{M}^2})
  \arrow[rr, yshift=1.5ex,"\sigma^{2}_{1}"]\arrow[rr, yshift=-1.5ex,"\tau^{2}_{1}"{below}]
 &&\mathbb{M}\text{od}(\Theta^{\infty}_{\mathbb{M}^1})
 \arrow[rr, yshift=1.5ex,"\sigma^{1}_{0}"]\arrow[rr, yshift=-1.5ex,"\tau^{1}_{0}"{below}]  
&&\mathbb{M}\text{od}(\Theta^{\infty}_{\mathbb{M}^0})
\end{tikzcd}$$

Thus if $\tau$ is a $\Theta^{\infty}_{\mathbb{M}^2}$-models in $\E$ then 
$F:=\sigma^{2}_{1}(\tau)\in\mathbb{M}\text{od}(\Theta^{\infty}_{\mathbb{M}^1})$ is the domain of 
$\tau$ and $H:=\tau^{2}_{1}(\tau)\in\mathbb{M}\text{od}(\Theta^{\infty}_{\mathbb{M}^1})$ is the codomain of 
$\tau$. Both $F$ and $H$ are globular weak $\infty$-functors. Also
$X:=\sigma^{2}_{0}(\tau)\in\mathbb{M}\text{od}(\Theta^{\infty}_{\mathbb{M}^0})$ and 
$Z:=\tau^{2}_{0}(\tau)\in\mathbb{M}\text{od}(\Theta^{\infty}_{\mathbb{M}^0})$ are the underlying
globular weak $\infty$-categories of $\tau$. Thus for shorter notation we write:

$$\begin{tikzcd}
X\arrow[r, bend left=50, ""{name=F}{above},"F"]
\arrow[r, bend right=50, ""{name=H}{below},"H"{below}]
&Z
\arrow[Rightarrow,"\tau"{left}, from=F,to=H]
\end{tikzcd}$$

such globular weak $(2,\infty)$-natural transformation.

\begin{proposition}
With the notation above, if $\tau$ is a $\Theta^{\infty}_{\mathbb{M}^2}$-models in $\E$
of dimension $2$, then it has a structure of pseudo-$2$-natural transformation.
\end{proposition}

\begin{proof}

We need to exhibit a commutative diagram:

$$\begin{tikzcd}
H_1(b)\circ^{1}_{0}(H_1(a)\circ^{1}_{0}\tau_1(x))\arrow[d,"{a(H_1(b),H_1(a),\tau_1(x))}"{left}]\arrow[rr,"{1_{H_1(b)}\circ^{2}_{0}\omega(a)}"]
&&H_1(b)\circ^{1}_{0}(\tau_1(y)\circ^{1}_{0}F_1(a))\arrow[d,,"{a(H_1(b),\tau_1(y),F_1(a))}"]\\
(H_1(b)\circ^{1}_{0}H_1(a))\circ^{1}_{0}\tau_1(x)\arrow[d,"{d_1(a,b)\circ^{2}_{0}1_{\tau_1(x)}}"{left}]
&&(H_1(b)\circ^{1}_{0}\tau_1(y))\circ^{1}_{0}F_1(a)\arrow[d,"{\omega(b)\circ{2}_{0}1_{F_1(a)}}"]\\
H_1(b\circ^{1}_{0}a)\circ^{1}_{0}\tau_1(x)\arrow[d,"{\omega(b\circ{1}_{0}a)}"{left}]
&&(\tau_1(z)\circ^{1}_{0}F_1(b))\circ^{1}_{0}F_1(a)\arrow[d,"{a(\tau_1(z),F_1(b),F_1(a))}"]\\
\tau_1(z)\circ^{1}_{0}F_1(b\circ^{1}_{0}a)
&&\tau_1(z)\circ^{1}_{0}(F_1(b)\circ^{1}_{0}F_1(a))\arrow[ll,"{1_{\tau_1(z)\circ^{2}_{0}d_0(b,a)}}"{below}]
\end{tikzcd}$$

for $1$-cells 
\begin{tikzcd}
x\arrow[r,"a"]&y\arrow[r,"b"]&z
\end{tikzcd}
of $X$, where are involved the coherence $2$-cells:

$$\begin{tikzcd}
a\circ^{1}_{0}(b\circ^{1}_{0}c)\arrow[r,"{a(a,b,c)}"]&(a\circ^{1}_{0}b)\circ^{1}_{0}c
\end{tikzcd}\qquad\begin{tikzcd}
 H_1(a)\circ^{1}_{0}\tau(x)\arrow[r,"{\omega(a)}"]&\tau(y)\circ^{1}_{0}F_1(a) 
 \end{tikzcd}$$
 
 and:
 
$$\begin{tikzcd}
F_1(b)\circ^{1}_{0}F_1(a)\arrow[r,"{d_0(b,a)}"]&F_1(b\circ^{1}_{0}a)
\end{tikzcd}\qquad\begin{tikzcd}
H_1(b)\circ^{1}_{0}H_1(a)\arrow[r,"{d_1(b,a)}"]&H_1(b\circ^{1}_{0}a)
\end{tikzcd}$$

In fact this diagram is the realization in $\E$, through the presheaf $\tau$, of some conglomerate of 
operations living in $\Theta^{\infty}_{\mathbb{M}^2}$. We are going to describe each operations 
which underlies expressions in this diagram, and this shall lead to the conglomerate of operations
leading, by passing to $\E$-models, to this diagram. 

In order to organize our computations we write:

$$\begin{tikzcd}
A\arrow[d,"(AG)"{left}]\arrow[rrr,"(AB)"]
&&&B\arrow[d,,"(BC)"]\\
G\arrow[d,"(GH)"{left}]
&&&C\arrow[d,"(CD)"]\\
H\arrow[d,"(HF)"{left}]
&&&D\arrow[d,"(DE)"]\\
F
&&&E\arrow[lll,"(EF)"{below}]
\end{tikzcd}$$

\begin{itemize}
\item Operation for $(AB)$:

The operations: 

$$\begin{tikzcd}
h_1(1(1))\star^{1}_{0}\tau(1(0))\arrow[r,"{H_1\star^{1}_{0}\tau}"]&2(1)\star^{1}_{0}2(1)\arrow[r,"\nu^{1}_0"]&2(1),
\end{tikzcd}\qquad\begin{tikzcd}
\tau(1(0))\star^{1}_{0}f_1(1(1))\arrow[r,"{\tau\star^{1}_{0}F_1}"]&2(1)\star^{1}_{0}2(1)\arrow[r,"\nu^{1}_0"]&2(1),
\end{tikzcd}$$

are such that $h_1(1(1))\star^{1}_{0}\tau(1(0))=\tau(1(0))\star^{1}_{0}f_1(1(1))$, also they are parallels, thus lead
to the operation $\omega$:

$$\begin{tikzcd}
2(2)\arrow[rrrdd,"\omega"]\\\\
2(1)\arrow[uu,xshift=1.1ex,"t^{2}_{1}"{right}]\arrow[uu,xshift=-1.1ex,"s^{2}_{1}"]\arrow[rrr, yshift=1.1ex,"{(H_1\star^{1}_{0}\tau)\circ\nu^{1}_{0}}"]
\arrow[rrr, yshift=-1.1ex,"{(\tau\star^{1}_{0}F_1)\circ\nu^{1}_{0}}"{below}]&&&h_1(1(1))\star^{1}_{0}\tau(1(0))
\end{tikzcd}$$

also we have the operation $[H_1;H_1]$:

$$\begin{tikzcd}
2(2)\arrow[rrrdd,"{[H_1;H_1]}"]\\\\
2(1)\arrow[uu,xshift=1.1ex,"t^{2}_{1}"{right}]\arrow[uu,xshift=-1.1ex,"s^{2}_{1}"]
\arrow[rrr, yshift=1.1ex,"H_1"]
\arrow[rrr, yshift=-1.1ex,"H_1"{below}]&&&h_1(1(1))
\end{tikzcd}$$

thus this leads to the operation:

$$\begin{tikzcd}
2(2)\arrow[r,"\nu^{2}_{0}"]&2(2)\star^{2}_{0}2(2)\arrow[rr,"{[H_1;H_1]\star^{2}_{0}\omega}"]&&
h_1(1(1))\star^{1}_{0}h_1(1(1))\star^{1}_{0}\tau(1(0))
\end{tikzcd}$$

which gives the operation $(AB)$:

$$\begin{tikzcd}
2(2)\arrow[rrr,"{([H_1;H_1]\star^{2}_{0}\omega)\circ\nu^{2}_{0}}"]&&&h_1(1(1))\star^{1}_{0}h_1(1(1))\star^{1}_{0}\tau(1(0))
\end{tikzcd}$$

\item Operation for $(BC)$:

The coherence for associativity is builds as follow:

$$\begin{tikzcd}
2(2)\arrow[rrrdd,"a"]\\\\
2(1)\arrow[uu,xshift=1.1ex,"t^{2}_{1}"{right}]\arrow[uu,xshift=-1.1ex,"s^{2}_{1}"]
\arrow[rrr, yshift=1.1ex,"{(\nu^{1}_{0}\star^{1}_{0}1_{2(1)})\circ\nu^{1}_{0}}"]
\arrow[rrr, yshift=-1.1ex,"{(1_{2(1)}\star^{1}_{0}\nu^{1}_{0})\circ\nu^{1}_{0}}"{below}]&&&2(1)\star^{1}_{0}2(1)\star^{1}_{0}2(1)
\end{tikzcd}$$

thus by precomposing it with:

$$\begin{tikzcd}
2(1)\star^{1}_{0}2(1)\star^{1}_{0}2(1)\arrow[rrr,"{H_1\star^{1}_{0}\tau\star^{1}_{0}F_1}"]&&&h_1(1(1))\star^{1}_{0}\tau(1(0))\star^{1}_{0}f_1(1(1))
\end{tikzcd}$$

we get the operation $(BC)$:

$$\begin{tikzcd}
2(2)\arrow[rrr,"{(H_1\star^{1}_{0}\tau\star^{1}_{0}F_1)\circ a}"]&&&h_1(1(1))\star^{1}_{0}\tau(1(0))\star^{1}_{0}f_1(1(1))
\end{tikzcd}$$

\item Operation for $(CD)$: this is similar to the operation $(AB)$ and we obtain:

$$\begin{tikzcd}
2(2)\arrow[rrr,"{(\omega\star^{2}_{0}[F_1;F_1])\circ\nu^{2}_{0}}"]&&&\tau(1(0))\star^{1}_{0}f_1(1(1))\star^{1}_{0}f_1(1(1))
\end{tikzcd}$$

\item Operation $(DE)$: this is similar to the operation $(BC)$ and we obtain:

$$\begin{tikzcd}
2(2)\arrow[rrr,"{(\tau\star^{1}_{0}F_1\star^{1}_{0}F_1)\circ a}"]&&&\tau(1(0))\star^{1}_{0}f_1(1(1))\star^{1}_{0}f_1(1(1))
\end{tikzcd}$$

\item Operation $(EF)$:

the operations:

$$\begin{tikzcd}
f_1(1(1))\star^{1}_{0}f_1(1(1))\arrow[rr,"{F_1\star^{1}_{0}F_1}"]&&2(1)\star^{1}_{0}2(1)\arrow[r,"\nu^{1}_{0}"]&2(1),
\end{tikzcd}\qquad\begin{tikzcd}
f_1(1(1))\star^{1}_{0}f_1(1(1))\arrow[rr,"{f_1(\mu^{1}_{0})}"]&&f_1(1(1))\arrow[r,"F_1"]&2(1),
\end{tikzcd}$$

are parallels, thus lead to the operation $d_0$:

$$\begin{tikzcd}
2(2)\arrow[rrrdd,"d_0"]\\\\
2(1)\arrow[uu,xshift=1.1ex,"t^{2}_{1}"{right}]\arrow[uu,xshift=-1.1ex,"s^{2}_{1}"]
\arrow[rrr, yshift=1.1ex,"{(F_1\star^{1}_{0}F_1)\circ\nu^{1}_{0}}"]
\arrow[rrr, yshift=-1.1ex,"{f_1(\mu^{1}_{0})\circ\nu^{1}_{0}}"{below}]&&&2(1)\star^{1}_{0}2(1)\star^{1}_{0}2(1)
\end{tikzcd}$$

Also we have the lifting $[\tau;\tau]$:

$$\begin{tikzcd}
2(2)\arrow[rrrdd,"{[\tau;\tau]}"]\\\\
2(1)\arrow[uu,xshift=1.1ex,"t^{2}_{1}"{right}]\arrow[uu,xshift=-1.1ex,"s^{2}_{1}"]
\arrow[rrr, yshift=1.1ex,"\tau"]
\arrow[rrr, yshift=-1.1ex,"\tau"{below}]&&&\tau(1(0))
\end{tikzcd}$$

thus this leads to:

$$\begin{tikzcd}
2(2)\arrow[r,"\nu^{2}_{0}"]&2(2)\star^{2}_{0}2(2)\arrow[rr,"{[\tau;\tau]\star^{2}_{0}d_0}"]&&\tau(1(0))\star^{1}_{0}f_1(1(1))\star^{1}_{0}f_1(1(1))
\end{tikzcd}$$

thus to the operation $(EF)$:

$$\begin{tikzcd}
2(2)\arrow[rrr,"{([\tau;\tau]\star^{2}_{0}d_0)\circ\nu^{2}_{0}}"]&&&\tau(1(0))\star^{1}_{0}f_1(1(1))\star^{1}_{0}f_1(1(1))
\end{tikzcd}$$

\item Operation $(AG)$: this is similar to the operations $(BC)$ and $(DE)$ and we obtain:

$$\begin{tikzcd}
2(2)\arrow[rrr,"{(H_1\star^{1}_{0}H_1\star^{1}_{0}\tau)\circ a}"]&&&h_1(1(1))\star^{1}_{0}f_1(1(1))\star^{1}_{0}\tau(1(0))
\end{tikzcd}$$

\item Operation $(GH)$: this is similar to the operation $(EF)$ and we obtain:

$$\begin{tikzcd}
2(2)\arrow[rrr,"{(d_1\star^{2}_{0}[\tau;\tau])\circ\nu^{2}_{0}}"]&&&h_1(1(1))\star^{1}_{0}h_1(1(1))\star^{1}_{0}\tau(1(0))
\end{tikzcd}$$

\item Operation $(HF)$:

we use the operations:

$$\begin{tikzcd}
h_1(1(1))\star^{1}_{0}h_1(1(1))\star^{1}_{0}\tau(1(0))\arrow[rrr,"{(H_1\circ h_1(\mu^{1}_{0}))\star^{1}_{0}\tau}"]
&&&2(1)\star^{1}_{0}2(1)\arrow[r,"\nu^{1}_{0}"]&2(1)
\end{tikzcd}$$

and

$$\begin{tikzcd}
\tau(1(0))\star^{1}_{0}f_1(1(1))\star^{1}_{0}f_1(1(1))\arrow[rrr,"{\tau\star^{1}_{0}(F_1\circ f_1(\mu^{1}_{0}))}"]
&&&2(1)\star^{1}_{0}2(1)\arrow[r,"\nu^{1}_{0}"]&2(1)
\end{tikzcd}$$

also we have the equality $h_1(1(1))\star^{1}_{0}h_1(1(1))\star^{1}_{0}\tau(1(0))=\tau(1(0))\star^{1}_{0}f_1(1(1))\star^{1}_{0}f_1(1(1))$,
and these operations are parallels, thus we get the lifting $\omega(\mu^{1}_{0})$:

$$\begin{tikzcd}
2(2)\arrow[rrrrrddd,"{\omega(\mu^{1}_{0})}"]\\\\\\
2(1)\arrow[uuu,xshift=1.1ex,"t^{2}_{1}"{right}]\arrow[uuu,xshift=-1.1ex,"s^{2}_{1}"]
\arrow[rrrrr, yshift=1.1ex,"{\left((H_1\circ h_1(\mu^{1}_{0}))\star^{1}_{0}\tau\right)\circ\nu^{1}_{0}}"]
\arrow[rrrrr, yshift=-1.1ex,"{\left(\tau\star^{1}_{0}(F_1\circ f_1(\mu^{1}_{0}))\right)\circ\nu^{1}_{0}}"{below}]
&&&&&\tau(1(0))\star^{1}_{0}f_1(1(1))\star^{1}_{0}f_1(1(1))
\end{tikzcd}$$

\item We have the equalities:

\begin{align*}
h_1(1(1))\star^{1}_{0}h_1(1(1))\star^{1}_{0}\tau(1(0)) &= h_1(1(1))\star^{1}_{0}\tau(1(0))\star^{1}_{0}f_1(1(1)) \\
        &= \tau(1(0))\star^{1}_{0}f_1(1(1))\star^{1}_{0}f_1(1(1)) \\
\end{align*}

which shows that operations $(AB)$, $(BC)$, $(CD)$, $(DE)$, $(EF)$, $(AG)$, $(GH)$ and $(HF)$ have the same arities, say:

$$\tau(1(0))\star^{1}_{0}f_1(1(1))\star^{1}_{0}f_1(1(1))$$

\item With the operations $(AB)$, $(BC)$, $(CD)$, $(DE)$ and $(EF)$ we have the following morphism in $\Theta^{\infty}_{\mathbb{M}^2}$:

$$\begin{tikzcd}
2(2)\star^{2}_{1}2(2)\star^{2}_{1}2(2)\star^{2}_{1}2(2)\star^{2}_{1}2(2)
\arrow[dd,"{\left(([H_1;H_1]\star^{2}_{0}\omega)\circ\nu^{2}_{0}\right)
\star^{2}_{1}\left((H_1\star^{1}_{0}\tau\star^{1}_{0}F_1)\circ a\right)\star^{2}_{1}
\left((\omega\star^{2}_{0}[F_1;F_1])\circ\nu^{2}_{0}\right)\star^{2}_{1}\left((\tau\star^{1}_{0}F_1\star^{1}_{0}F_1)\circ a\right)
\star^{2}_{1}\left(([\tau;\tau]\star^{2}_{0}d_0)\circ\nu^{2}_{0}\right)}"]\\\\
\tau(1(0))\star^{1}_{0}f_1(1(1))\star^{1}_{0}f_1(1(1))
\end{tikzcd}$$

From it we get the following operation $(ABCDEF)$ in $\Theta^{\infty}_{\mathbb{M}^2}$:

$$\begin{tikzcd}
2(2)\arrow[d,"{\nu^{2}_{1}}"]\\
2(2)\star^{2}_{1}2(2)\arrow[d,"{1_{2(2)}\star^{2}_{1}\nu^{2}_{1}}"]\\
2(2)\star^{2}_{1}2(2)\star^{2}_{1}2(2)\arrow[d,"{1_{2(2)}\star^{2}_{1}1_{2(2)}\star^{2}_{1}\nu^{2}_{1}}"]\\
2(2)\star^{2}_{1}2(2)\star^{2}_{1}2(2)\star^{2}_{1}2(2)\arrow[d,"{1_{2(2)}\star^{2}_{1}1_{2(2)}\star^{2}_{1}1_{2(2)}\star^{2}_{1}\nu^{2}_{1}}"]\\
2(2)\star^{2}_{1}2(2)\star^{2}_{1}2(2)\star^{2}_{1}2(2)\star^{2}_{1}2(2)
\arrow[dd,"{\left(([H_1;H_1]\star^{2}_{0}\omega)\circ\nu^{2}_{0}\right)
\star^{2}_{1}\left((H_1\star^{1}_{0}\tau\star^{1}_{0}F_1)\circ a\right)\star^{2}_{1}
\left((\omega\star^{2}_{0}[F_1;F_1])\circ\nu^{2}_{0}\right)\star^{2}_{1}\left((\tau\star^{1}_{0}F_1\star^{1}_{0}F_1)\circ a\right)
\star^{2}_{1}\left(([\tau;\tau]\star^{2}_{0}d_0)\circ\nu^{2}_{0}\right)}"]\\\\
\tau(1(0))\star^{1}_{0}f_1(1(1))\star^{1}_{0}f_1(1(1))
\end{tikzcd}$$

\item With the operations $(AG)$, $(GH)$ and $(HF)$ we have the following morphism in $\Theta^{\infty}_{\mathbb{M}^2}$:

$$\begin{tikzcd}
2(2)\star^{2}_{1}2(2)\star^{2}_{1}2(2)
\arrow[dd,"{\left((H_1\star^{1}_{0}H_1\star^{1}_{0}\tau)\circ a\right)
\star^{2}_{1}\left((d_1\star^{2}_{0}[\tau;\tau])\circ\nu^{2}_{0}\right)\star^{2}_{1}
\left(\omega(\mu^{1}_{0})\right)}"]\\\\
\tau(1(0))\star^{1}_{0}f_1(1(1))\star^{1}_{0}f_1(1(1))
\end{tikzcd}$$

and from it we get the following operation $(AGHF)$ in $\Theta^{\infty}_{\mathbb{M}^2}$:

$$\begin{tikzcd}
2(2)\arrow[d,"{\nu^{2}_{1}}"]\\
2(2)\star^{2}_{1}2(2)\arrow[d,"{1_{2(2)}\star^{2}_{1}\nu^{2}_{1}}"]\\
2(2)\star^{2}_{1}2(2)\star^{2}_{1}2(2)
\arrow[dd,"{\left((H_1\star^{1}_{0}H_1\star^{1}_{0}\tau)\circ a\right)
\star^{2}_{1}\left((d_1\star^{2}_{0}[\tau;\tau])\circ\nu^{2}_{0}\right)\star^{2}_{1}
\left(\omega(\mu^{1}_{0})\right)}"]\\\\
\tau(1(0))\star^{1}_{0}f_1(1(1))\star^{1}_{0}f_1(1(1))
\end{tikzcd}$$

\end{itemize}

The operations $(ABCDEF)$ and $(AGHF)$ are parallels, thus lead to the lifting:

$$\begin{tikzcd}
2(3)\arrow[rrrrrddd,"{[(ABCDEF);(AGHF)]}"]\\\\\\
2(2)\arrow[uuu,xshift=1.1ex,"t^{3}_{2}"{right}]\arrow[uuu,xshift=-1.1ex,"s^{3}_{2}"]
\arrow[rrrrr, yshift=1.1ex,"{(ABCDEF)}"]
\arrow[rrrrr, yshift=-1.1ex,"{(AGHF)}"{below}]
&&&&&\tau(1(0))\star^{1}_{0}f_1(1(1))\star^{1}_{0}f_1(1(1))
\end{tikzcd}$$

which is a $3$-operation in $\Theta^{\infty}_{\mathbb{M}^2}$. But $\tau$ has dimension $2$, thus the $\E$-realization of 
the operations $(ABCDEF)$ and $(AGHF)$ agree, which gives the commutative diagram
for pseudo-$2$-natural transformations.
\end{proof}

\section{The globular weak $\infty$-category of globular weak $\infty$-categories in the sense of Grothendieck}

Thanks to the result in \cite{bourke-injectif} it is possible to use the formalism of globular operads in \cite{batanin-main},
at least for the globular free contractible operad $\mathbb{B}^{0}_{\text{C}}$ which algebras are globular weak $\infty$-categories of
Batanin. The main result in \cite{bourke-injectif} solves a conjecture in \cite{ara-these} which basically claimed 
the equivalence of categories: 

$$\mathbb{M}\text{od}(\Theta^{\infty}_{\mathbb{M}^0})\simeq\mathbb{B}^{0}_{\text{C}}\text{-}\mathbb{A}\text{lg}$$

that is, the approach of globular weak $\infty$-categories of Batanin and Grothendieck are equivalent. 

In this article we built the following coglobular object in $\Cat$:

$$\begin{tikzcd}
\Theta^{\infty}_{\mathbb{M}^0}
  \arrow[rr, yshift=1.5ex,"s^{0s^{1}_{0}}_{1}"]\arrow[rr, yshift=-1.5ex,"t^{1}_{0}"{below}]
 &&\Theta^{\infty}_{\mathbb{M}^1}
 \arrow[rr, yshift=1.5ex,"s^{2}_{1}"]\arrow[rr, yshift=-1.5ex,"t^{2}_{1}"{below}]  
&&\Theta^{\infty}_{\mathbb{M}^2}
\arrow[rr, yshift=1.5ex,"s^{3}_{2}"]\arrow[rr, yshift=-1.5ex,"t^{3}_{2}"{below}]
&&\Theta^{\infty}_{\mathbb{M}^3}
\arrow[rr, yshift=1.5ex,"s^{4}_{3}"]\arrow[rr, yshift=-1.5ex,"t^{4}_{3}"{below}]
&&\Theta^{\infty}_{\mathbb{M}^4}\cdots\Theta^{\infty}_{\mathbb{M}^{n-1}}
\arrow[rr, yshift=1.5ex,"s^{n}_{n-1}"]\arrow[rr, yshift=-1.5ex,"t^{n}_{n-1}"{below}]
&&\Theta^{\infty}_{\mathbb{M}^n}\cdots 
\end{tikzcd}$$

If this coglobular object $\Theta^{\infty}_{\mathbb{M}^{\bullet}}$ is a $\mathbb{B}^{0}_{\text{C}}$-coalgebra which means that we have a 
coaction:

$$\begin{tikzcd}
\mathbb{B}^{0}_{\text{C}}\arrow[rr]&&\mathbb{C}\text{oend}(\Theta^{\infty}_{\mathbb{M}^{\bullet}})
\end{tikzcd}$$

where $\mathbb{C}\text{oend}(\Theta^{\infty}_{\mathbb{M}^{\bullet}})$ is the coendomorphism operad
associated to $\Theta^{\infty}_{\mathbb{M}^{\bullet}}$ (this exists because $\Cat$ has all small colimits),
then this shows that globular weak $(n,\infty)$-transformations, for all $n\in\mathbb{N}$, by coherators,
organize in a globular weak $\infty$-categories in the sense of Batanin, thus in the sense of Grothendieck. 
Thus the $\mathbb{B}^{0}_{\text{C}}$-coalgebraicity of $\Theta^{\infty}_{\mathbb{M}^{\bullet}}$ means
that globular weak $\infty$-categories of Grothendieck organize in a globular weak $\infty$-categories
of Grothendieck. This fact is a necessary step for an accurate approach of globular weak $\infty$-stacks
as wished by Alexandre Grothendieck in \cite{grothendieck}.

A way to prove this coalgebraicity is to show that the operad $\mathbb{C}\text{oend}(\Theta^{\infty}_{\mathbb{M}^{\bullet}})$
is contractible and equipped with a composition system. We shall give a complete syntactical proof of this coalgebraicity in
a future paper.

\bigbreak{}
\begin{minipage}{1.0\linewidth}
Laboratoire de Math\'ematiques d'Orsay, UMR 8628\\
Universit\'e de Paris-Saclay and CNRS\\
B\^atiment 307, Facult\'e des Sciences d'Orsay\\
94015 ORSAY Cedex, FRANCE\\
\href{mailto:camell.kachour@universite-paris-saclay.fr}{\url{camell.kachour@universite-paris-saclay.fr}}
\end{minipage}
\end{document}